%% file: main_arxiv.tex
\documentclass[journal]{IEEEtran}

\makeatletter
\let\NAT@parse\undefined
\makeatother

\input{Preamble/packages}
\input{Preamble/macros}
\input{Preamble/glossaries}
\input{Preamble/tikz_utils}

\begin{document}
\bstctlcite{BSTcontrol}

\title{Compositional Online Learning for Multi-Objective System Co-Design}

\author{Meshal Alharbi, Munther A. Dahleh, Gioele Zardini
\thanks{The authors are with the Laboratory for Information \& Decision Systems,
Massachusetts Institute of Technology, Cambridge, MA 02139 {\tt \{meshal,dahleh,gzardini\}@mit.edu}
}
\thanks{This material is based upon work supported by the Defense Advanced Research Projects Agency (DARPA) under Award No. D25AC00373. The views and conclusions contained in this document are those of the authors and should not be interpreted as representing the official policies, either expressed or implied, of the U.S. Government.}
}

\maketitle

\input{Sections/abstract}
\input{Sections/introduction}

\input{Sections/literature}
\input{Sections/preliminaries}
\input{Sections/online_design}
\input{Sections/optimistic_elimination}
\input{Sections/online_codesign}
\input{Sections/case_study}

\input{Sections/conclusion}

\bibliographystyle{IEEEtran}
\bibliography{references}

\end{document}

%% file: Preamble/packages.tex
\usepackage[x11names]{xcolor}
\usepackage{cite}
\usepackage{amsmath, amsfonts, amssymb, amsthm, mathtools, dsfont, circledsteps, bm}
\usepackage{algorithm}
\usepackage{algpseudocodex}
\usepackage[hidelinks]{hyperref}
\hypersetup{
    colorlinks=false,
    linkcolor=Blue3,
    citecolor=Blue3,
    urlcolor=violet,
    filecolor=red,
}
\usepackage{cleveref}
\usepackage{xspace}
\usepackage{booktabs, nicematrix}
\usepackage{adjustbox}



%% file: Preamble/macros.tex
\definecolor{dpred}{rgb}{0.7, 0.0, 0.0}
\definecolor{dpgreen}{rgb}{0.0, 0.5, 0.0}
\definecolor{dppurple}{rgb}{0.5,0,0.5}
\newcommand{\F}[1]{\textcolor{dpgreen}{#1}}
\newcommand{\R}[1]{\textcolor{dpred}{#1}}
\newcommand{\I}[1]{\textcolor{DarkOrange2}{#1}}

\newcommand{\funPosetF}{{\F{\mathcal{F}}}}
\newcommand{\resPosetR}{{\R{\mathcal{R}}}}
\newcommand{\impSetI}{\mathcal{I}}
\newcommand{\op}{{\mathrm{op}}}

\newcommand{\opt}{{\mathrm{opt}}}

\newcommand{\dprb}{{\mathrm{dp}}}
\newcommand{\cdprb}{{\mathrm{cdp}}}
\newcommand{\dpOf}[2]{\mathsf{DP}\{#1, #2\}}
\newcommand{\dpiOf}[3]{\mathsf{DPI}\{#1, #2, #3\}}

\newcommand{\prov}{\F{\mathsf{prov}}}
\newcommand{\req}{\R{\mathsf{req}}}
\newcommand{\provImp}[1]{\prov(\I{#1})}
\newcommand{\reqImp}[1]{\req(\I{#1})}
\newcommand{\reqImpStar}[1]{\req(\I{#1}^\star)}
\newcommand{\provImpStar}[1]{\prov(\I{#1}^\star)}

\newcommand{\provOpt}{\F{\mathsf{prov}}^\mathrm{opt}}
\newcommand{\provPes}{\F{\mathsf{prov}}^\mathrm{pes}}
\newcommand{\reqOpt}{\R{\mathsf{req}}^\mathrm{opt}}
\newcommand{\reqPes}{\R{\mathsf{req}}^\mathrm{pes}}

\newcommand{\FixFunMinResOf}[1]{\mathsf{FixFunMinRes}(\F{#1})}

\newcommand{\OptAnti}{A_{\F{f}}}
\newcommand{\SubAnti}{\hat{A}_{\F{f}}}

\newcommand{\impPointIqJ}{\tupII{\I{i}_q}{\I{j}}}

\newcommand{\resref}{\R{r}^\mathrm{ref}}

\newcommand{\HVOf}[1]{\mathsf{HV}(#1)}
\newcommand{\HVD}{\mathsf{HVD}}
\newcommand{\HVDOf}[2]{\mathsf{HVD}(#1, #2)}

\newcommand{\hist}{\mathcal{H}}

\newcommand{\HALTON}{\textsc{Halton}\xspace}

\newcommand{\NSGAIII}{\textsc{NSGA-III}\xspace}
\newcommand{\MOEAD}{\textsc{MOEA/D}\xspace}
\newcommand{\RVEA}{\textsc{RVEA}\xspace}
\newcommand{\KGB}{\textsc{KGB}\xspace}
\newcommand{\qLogEHVI}{\textsc{qLogEHVI}\xspace}
\newcommand{\OPTELIM}{\textbf{\textsc{Ours}}\xspace}

\newcommand{\TableVerStretch}{1.1}
\newcommand{\ul}[1]{\underline{#1}}

\newcommand{\set}[1]{\ensuremath{\{ #1 \}}}
\newcommand{\card}[1]{\ensuremath{\lvert #1 \rvert}}
\newcommand{\tup}[1]{\ensuremath{\langle #1 \rangle}}
\newcommand{\tupII}[2]{\ensuremath{\langle #1, #2 \rangle}}
\newcommand{\tupIII}[3]{\ensuremath{\langle #1, #2, #3 \rangle}}

\DeclareMathOperator*{\anti}{\mathsf{Anti}}

\DeclareMathOperator*{\pow}{\mathsf{Pow}}
\newcommand{\antiOf}[1]{\mathsf{Anti}_{#1}}
\newcommand{\solveOf}[1]{\mathsf{Solve}_{#1}}
\newcommand{\powOf}[1]{\pow(#1)}

\NewCommandCopy{\ordinaryforall}{\forall}
\NewCommandCopy{\ordinaryexists}{\exists}
\NewCommandCopy{\ordinarynexists}{\nexists}
\RenewDocumentCommand{\forall}{}{\mathop{{}\ordinaryforall}}
\RenewDocumentCommand{\exists}{}{\mathop{{}\ordinaryexists}}
\RenewDocumentCommand{\nexists}{}{\mathop{{}\ordinarynexists}}

\newcommand{\upper}{\operatorname{\uparrow}}

\newcommand\numberthis{\addtocounter{equation}{1}\tag{\theequation}}

\newenvironment{talign*}
 {\csname align*\endcsname}
 {\endalign}

\theoremstyle{definition}
\newtheorem{definition}{Definition}[section]

\newtheorem{assumption}{Assumption}

\crefname{assumption}{assumption}{assumptions}
\Crefname{assumption}{Assumption}{Assumptions}

\newtheorem{problem}{Problem}[section]

\newtheorem{theorem}{Theorem}[section]
\newtheorem{corollary}{Corollary}[section]

\newtheorem{proposition}{Proposition}[section]

\newtheorem{example}{Example}[section]

\newtheorem{remark}{Remark}[section]

\usepackage{enumitem}
\newtheorem{lemma}{Lemma}[section]

\Crefname{lemma}{Lemma}{Lemmas}
\Crefname{lemenumi}{Lemma}{Lemmas}
\AtBeginEnvironment{lemma}{%
    \crefalias{enumi}{lemenumi}%
    \setlist[enumerate,1]{
        label={\textit{(\alph*)}},
        ref={\thelemma.\alph*}
    }%
}

%% file: Preamble/tikz_utils.tex
\usetikzlibrary{positioning,intersections, hobby, patterns, calc, decorations.pathmorphing, decorations.markings, shadows,shapes, cd,arrows.meta, fit,quotes}

\tikzset{
   tick/.style={postaction={
      decorate,
      decoration={markings, mark=at position 0.5 with {\draw[-] (0,.4ex) -- (0,-.4ex);}}}
   }
}
\tikzstyle{block} = [draw, rectangle, minimum height=2em, minimum width=3em,font=\bfseries,rounded corners,thick]
\tikzstyle{block} = [draw, rectangle, minimum height=2em, minimum width=3em]
\tikzstyle{block1} = [draw, rectangle, minimum height=1.5em, minimum width=2.5em]
\tikzstyle{blockDyn} = [draw, rectangle, minimum height=2.5em, minimum width=3.5em, align=center, inner sep=10pt, thick, fill=white, copy shadow={draw=black,fill=black,opacity=1,shadow xshift=0.5ex,shadow yshift=-0.5ex}]
\tikzstyle{blockAlg} = [draw, rectangle, minimum height=1.5em, minimum width=2.5em, align=center, inner sep=10pt, thick]
\tikzstyle{sum} = [draw,circle]

\tikzstyle{nodePre} = [circle, draw,inner sep=1pt,node contents={$\preceq$},thick]
\tikzstyle{nodePreEmpty} = [circle, draw,inner sep=1pt,thick]
\tikzstyle{nodePos} = [circle, draw,inner sep=1pt,node contents={$\posceq$},thick]
\tikzstyle{nodeProd} = [rectangle, draw,inner sep=4pt,node contents={$\times$},rounded corners,thick]
\tikzstyle{nodeSum} = [rectangle, draw,inner sep=4pt,node contents={$\mathbf{+}$},rounded corners,thick]

\definecolor{red}{rgb}{0.75, 0.0, 0.0}

\tikzset{
  dp port dot scale/.store in=\dpportdotscale,
  dp port dot scale=0.3
}
\tikzset{fcname/.store in =\fcname, fcname={}}
\tikzset{funame/.store in =\funame, funame={}}
\tikzset{rcname/.store in =\rcname, rcname={}}
\tikzset{runame/.store in =\runame, runame={}}
\tikzset{whereres/.store in =\whereres, whereres=0.5}
\tikzset{wherefun/.store in =\wherefun, wherefun=0.5}
\tikzset{relres/.store in =\relres, relres={above}}
\tikzset{relfun/.store in =\relfun, relfun={above}}
\tikzset{posres/.store in =\posres, posres=1}
\tikzset{posfun/.store in =\posfun, posfun=1}
\tikzset{loos/.store in =\loos, loos=2}
\tikzset{feedback/.store in =\feedback, feedback=0}
\tikzset{
   DP/.style={
      label/.style={
         font=\everymath\expandafter{\the\everymath\scriptstyle},
         inner sep=5pt,
         node distance=2pt and -2pt},
      semithick,
      node distance=1 and 1,
      rconn/.style={color=white,opacity=0.0,postaction={decorate}, shorten <=3.2pt, shorten >= 0.8,
      decoration={markings, 
      mark= at position 0 with {
               \coordinate (a);
      },
      mark=at position .5 with
      {
              \ifthenelse{\equal{\feedback}{1}}{\def\angleOut{90}\def\angleIn{90}}{\def\angleOut{0}\def\angleIn{180}}    
              \coordinate (b);
              \draw[dashed,dpred,opacity=1.0] (a) to[out=\angleOut,in=\angleIn,looseness=\loos] 
              node[pos=\posres,\relres=\whereres mm,dpred,opacity=1,fill=white,inner sep=1pt,outer sep=1pt]{\footnotesize{\rcname}} (b);
      },
      mark= at position 1 with 
      {
             \ifthenelse{\equal{\feedback}{1}}{\def\angleOut{0}\def\angleIn{0}}{\def\angleOut{180}\def\angleIn{0}} 
              \ifthenelse{\equal{\feedback}{1}}{\def\symbol{\succeq}}{\def\symbol{\preceq}} 
              \coordinate (c);
              \draw[dpgreen,opacity=1.0] (c) to[out=\angleOut,in=\angleIn,looseness=\loos]
              node[pos=\posfun,\relfun=\wherefun mm,dpgreen,opacity=1,fill=white,inner sep=1pt,outer sep=1pt]{\footnotesize{\fcname}} (b){}; 
              \node[draw,circle,inner sep=0.5pt,color=black,fill=white,opacity=1.0] at (b) (nodepreceq) {$\symbol$}; 
      }
      }},
      runconn/.style={color=dpred,dashed,postaction={decorate},
      decoration={markings,
      mark= at position 1 with {
              \coordinate (a);
              \draw[dpred,opacity=1.0,dashed] ($(a)+(0.05,0)$) --++ (0.5,0) node[\relres,pos=\posres]{\footnotesize{\runame}};}
      }
      },
      funconn/.style={color=white,postaction={decorate},
      decoration={markings,
      mark= at position 0 with {
      \coordinate (a);
      \draw[dpgreen] ($(a)+(-0.05,0)$) -- ($(a)+(-0.5,0)$) node[\relfun, pos=\posfun]{\footnotesize{\funame}};}
      }
      },
      execute at begin picture={\tikzset{
         x=\dpx, y=\dpy,
         every fit/.style={inner xsep=\dpx, inner ysep=\dpy}}}
      },
   dpx/.store in=\dpx,
   dpx = 1.5cm,
   dpy/.store in=\dpy,
   dpy = 1.5ex,
   dp port sep/.store in=\dpportsep,
   dp port sep=2,
   dp port length/.store in=\dpportlen,
   dp port length=4pt,
   dp min width/.store in=\dpminwidth,
   dp min width=0.5cm,
   dp rounded corners/.store in=\dpcorners,
   dp rounded corners=2pt,
   dp small/.style={dp port sep=1, dp port length=2.5pt, dpx=.4cm, dp min width=.4cm, dpy=.7ex},
   dp/.code 2 args={
      \pgfmathsetlengthmacro{\dpheight}{\dpportsep * (max(#1,#2)) * \dpy}
      \pgfkeysalso{draw,%
        minimum width=\dpminwidth,%
        minimum height=\dpheight,%
        font=\bfseries,
        outer sep=0pt,%
        inner sep=5pt,%
        rounded corners=\dpcorners,
        thick,
        prefix after command={\pgfextra{\let\fixname\tikzlastnode}},
        append after command={\pgfextra{\draw
            \ifnum #1=0{} \else foreach \i in {1,...,#1} { 
            ($(\fixname.north west)!{\i/(#1+1)}!(\fixname.south west)$) +(0,0) node[solid,left,circle,color=dpgreen,draw,fill=dpgreen,scale=\dpportdotscale] {} coordinate (\fixname_fun\i) -- +(0,0) coordinate (\fixname_fun\i')}\fi 
            \ifnum #2=0{} \else foreach \i in {1,...,#2} {
            ($(\fixname.north east)!{\i/(#2+1)}!(\fixname.south east)$) +(0,0) coordinate (\fixname_res\i') -- +(0,0) node[solid,right,circle,color=dpred,draw,fill=dpred,scale=\dpportdotscale] {} coordinate (\fixname_res\i)}\fi;
         }}}
         },
      dp name/.style={append after command={\pgfextra{\node[label=center,inner sep=2pt,fill=white] at (\fixname) {\textbf{#1}};}}}
   }

%% file: Sections/abstract.tex
\begin{abstract}
Many engineered systems must balance competing objectives, such as performance and safety, cost and reliability, or efficiency and sustainability, and are naturally modeled as compositions of interacting subsystems. 
We study online multi-objective decision-making in monotone co-design, where functionalities and resources are partially ordered, and the goal is to identify the target-feasible antichain of non-dominated trade-offs using few expensive evaluations. 
We introduce \emph{optimistic evaluators}: history-dependent bounds on functionality and resource mappings that enable safe elimination of implementations before full evaluation. 
Based on these evaluators, we develop an elimination-based rejection-sampling algorithm, prove its soundness, and show that the admissible region shrinks monotonically as information accumulates. 
We instantiate the framework under monotonicity, Lipschitz continuity, and linear-parametric structure. 
For compositional co-design problems modeled by multigraphs, we show how local optimistic certificates propagate through the tractable remainder of the graph to yield system-level optimistic feasibility and resource bounds. 
Experiments on multi-robot fleet design, intermodal mobility systems, and synthetic monotone and Lipschitz benchmarks show substantial sample-efficiency gains over uniform sampling, Bayesian optimization, and multi-objective evolutionary algorithms.
\end{abstract}

%% file: Sections/introduction.tex
\section{Introduction}
\label{sec:introduction}

Many engineered systems must balance competing criteria such as performance and safety, cost and reliability, or efficiency and sustainability.
In such settings, there is rarely a single optimal solution; instead, one must reason directly about a set of non-dominated trade-offs.
This perspective is becoming increasingly important in modern engineered systems, from robotics and autonomous mobility to society-critical infrastructure planning, where overall behavior emerges from the interaction of multiple subsystems and multiple stakeholders.
In these settings, optimizing one component at a time, or collapsing all objectives into a single scalar reward, can obscure critical trade-offs and lead to brittle or opaque system-level decisions~\cite{ehrgott2005multicriteria,zardini2023co}.

Although multi-objective formulations are natural in many applications, they are difficult to solve for at least three distinct reasons.
First, trade-offs typically arise from the coupling of multiple subsystems, which induces large search spaces and strong dependencies across objectives. 
Second, even when the full model is known, recovering the entire efficient set of solutions is generally computationally demanding;
for example, even in \glspl{acr:molp}, computational complexity depends strongly on what representation of the efficient set is sought and on the size of that representation~\cite{moitra2011pareto,dorfler2022benson}. 
Third, in black-box settings where the map from decisions to outcomes can only be queried, sequential methods may exhibit exponential computational scaling in the number of objectives, even in discrete domains~\cite{crepon2024sequential}. 
Taken together, these obstacles suggest that scalable online methods must explicitly exploit structural properties of the problem to enable efficient decision-making.

Over the past few years, monotone co-design theory has emerged as a powerful response to these challenges, especially for complex engineered systems built from many interacting parts~\cite{censi2015mathematical, zardini2021co, zardini2022task, censi2022, zardini2023co, stralz2026task}.
Its central advantage is \emph{compositionality}: each subsystem is described through interfaces of functionalities and resources, both modeled as partially ordered sets, and larger systems are obtained by composing smaller ones~\cite{zardini2023co}. 
This perspective is appealing for two reasons. 
On the modeling side, it preserves the native order structure of the quantities being traded, rather than forcing them into a fixed Euclidean representation. 
On the algorithmic side, compositionality enables system-level reasoning through structured operations on smaller blocks, making it possible to analyze rich trade-offs without flattening the problem into a single monolithic optimization.

At the same time, existing co-design methods typically assume that each atomic block is already available in explicit form, for instance, as a finite catalog of implementations or through a tractable model from which functionality--resource trade-offs can be computed efficiently. 
In many practical settings, however, the most consequential block is precisely the one that is not available in this form: it may only be accessible through expensive simulation, hardware-in-the-loop experiments, or another black-box evaluation. 
This gap motivates an \emph{adaptive sampling} approach: rather than exhaustively exploring the implementation space, the agent should focus its evaluations on the most informative candidates, leveraging whatever structural knowledge is available to prune away regions that provably cannot contribute to optimal trade-offs.

\subsubsection*{Statement of contribution}
This paper addresses that gap and formulates online multi-objective learning directly in the co-design setting, preserving the system's native order structure rather than imposing a global metric or scalarization. 
The overall system is treated compositionally and order-theoretically, and stronger assumptions are used only locally, so continuous and discrete components can coexist within the same compositional model. 
This yields a framework that learns expensive subsystems through the graph while targeting target-feasible antichains directly.

The main contributions of the paper are as follows.
First, we formulate an online multi-objective learning problem over partially ordered functionality and resource spaces, where the object to be recovered is the antichain of non-dominated resources compatible with a desired functionality target.
Second, we introduce generic optimistic and pessimistic evaluators, establish safe elimination and monotone shrinking properties, and derive an elimination-based rejection sampler that focuses evaluations on implementations that remain relevant to the current target-feasible frontier.
Third, we instantiate the evaluator template under several common forms of structure, including monotonicity, Lipschitz continuity, and linear-parametric models, showing how problem-specific regularity yields computable optimistic certificates.
Fourth, we extend the framework to multigraph co-design problems and show how local optimistic certificates for an expensive atomic block propagate through the tractable remainder of the graph to yield system-level optimistic feasibility and resource bounds.
Finally, we provide rigorous correctness guarantees for the elimination mechanism and evaluate the resulting algorithms on heterogeneous multi-agent fleet co-design, intermodal mobility co-design, and synthetic benchmarks with monotone and Lipschitz structure.

%% file: Sections/literature.tex
\section{Related Literature}
\label{sec:literature}

This work lies at the intersection of compositional system design, multi-objective optimization, and sequential decision-making.
We give a brief review of the literature in these fields.

\subsection{Multi-Objective Evolutionary Algorithms}

Designing complex engineered systems entails coordinated decisions across coupled subsystems, yielding multidisciplinary optimization problems that are often costly to evaluate and require methods able to handle black-box models and mixed or structured variables~\cite{ehrgott2005multicriteria, martins2013multidisciplinary, seshia2016design, zardini2023co}.
Multi-objective \glspl{acr:ea} are one of the dominant classes of methods used in practice for approximating Pareto fronts in such black-box settings~\cite{liu2023survey,li2024survey}.
At a high level, \glspl{acr:ea} maintain a \emph{population} of candidate solutions that is iteratively updated through selection, crossover, and mutation operators.
Two prevailing strategies for guiding the search toward the Pareto front are (i) \emph{dominance-based} methods, which use non-dominated sorting and crowding mechanisms (e.g., \cite{deb2013evolutionary}); and (ii) \emph{decomposition-based} methods, which convert the multi-objective problem into a collection of cooperatively solved scalar subproblems (e.g., \cite{zhang2007moea}).
A key concern in multi-objective optimization is how to measure the quality of an approximate Pareto front~\cite{li2019quality}, and \glspl{acr:ea} are commonly assessed using quality indicators such as hypervolume and epsilon-based metrics.

\subsection{Multi-Objective Bayesian Optimization}

In contrast to population-based \gls{acr:ea} methods, \gls{acr:bo} maintains an explicit probabilistic surrogate of the objective functions and uses it to select the next design to evaluate, allowing information from queried designs to generalize to unqueried ones~\cite{wang2023recent}.
Within this family, methods can be broadly organized by the acquisition criterion they employ:
\emph{Scalarization-based} methods reduce the problem to a single objective by randomizing over weight vectors~\cite{knowles2006parego};
\emph{Hypervolume-based} methods directly target the expected improvement in dominated hypervolume~\cite{daulton2020differentiable, ament2023unexpected};
\emph{Information-theoretic} methods select points that maximize the expected information gain about the Pareto front~\cite{hernandez2016predictive, belakaria2019max}.
Two especially relevant examples are $\epsilon$-PAL~\cite{zuluaga2016pal}, which eliminates designs using Gaussian process confidence bounds, and FlexiBO~\cite{iqbal2023flexibo}, which studies cost-aware decoupled objective evaluations. 
Our approach differs from standard \gls{acr:bo} by exploiting deterministic structural properties rather than posterior uncertainty and repeated global surrogate refitting.

\subsection{Multi-Objective Bandits}

Multi-objective bandits study sequential identification of Pareto-optimal arms from vector-valued stochastic rewards~\cite{drugan2013designing,busa2017multi}.
Unlike \gls{acr:bo}, they typically do not rely on a global surrogate model, and unlike \glspl{acr:ea}, they more strongly emphasize sample-complexity guarantees for sequential exploration.
A central idea in the bandits literature is the \emph{optimism under uncertainty} principle, in which confidence sets guide exploration~\cite{lattimore2020bandit}.
Classical fixed-confidence results give gap-dependent sample-complexity bounds for Pareto set identification~\cite{auer2016pareto}.
More recent sequential-testing approaches target asymptotically optimal instance-dependent sample complexity, but can introduce a nontrivial computational cost for solving the associated identification subproblems~\cite{crepon2024sequential}. 
Related work also extends Pareto identification beyond the standard componentwise order to more general preference cones, and to structured linear models where the effective difficulty is governed by a lower-dimensional representation rather than all arms independently~\cite{ararat2023vector,kone2025bandit}.

Our approach differs from prior work in two key ways. 
First, it operates on general \emph{partially ordered} functionality and resource spaces, without fixing a particular topology or performance metric. 
The overall system is treated order-theoretically, and additional structure, such as monotonicity or metric smoothness, is optional and enters only through specific optimistic evaluators.
Second, it treats system design as a \emph{compositional} multi-objective problem, using the co-design multigraph to determine which implementations must be actively explored and which can be tractably resolved.
This yields a modular framework that accommodates heterogeneous structure and confines online learning to the expensive subsystems.

%% file: Sections/preliminaries.tex
\section{Preliminaries}
\label{sec:preliminaries}

This section introduces mathematical preliminaries underlying the concepts used throughout the paper, together with a brief primer on system co-design.

\subsection{Background on Order}

We begin by recalling the order-theoretic notions that underpin our framework.

\begin{definition}[\glsdefhere{\Glsxtrshort}{acr:poset}]
A \textit{\glsxtrfull{acr:poset}} is a tuple $\mathcal{P} = \tupII{P}{\preceq_\mathcal{P}}$, where $P$ is a set and $\preceq_\mathcal{P}$ is a partial
order (a binary relation on~$P$ that is reflexive, antisymmetric, and transitive).
\end{definition}
For a \poset
$\mathcal{P} = \tupII{P}{\preceq_\mathcal{P}}$, we write $x \prec_\mathcal{P} y$ when
$x \preceq_\mathcal{P} y$ and $x \neq y$. A \emph{total order} is a \poset in which every pair of elements is comparable.

\begin{definition}[Opposite \gls{acr:poset}]
\label{def:opposite-poset}
The \emph{opposite} of a \gls{acr:poset}~$\mathcal{P} = \tup{ P,\preceq_\mathcal{P}}$ is the poset $\mathcal{P}^\op \coloneq \tupII{P}{ \preceq_{\mathcal{P}}^\op }$ with the same elements and reversed ordering:
$
x \preceq_{\mathcal{P}}^\op y \Leftrightarrow 
y \preceq_{\mathcal{P}} x
$.
\end{definition}

\begin{definition}[Product \glsxtrshort{acr:poset}]
Given two \glsxtrshortpl{acr:poset} $\mathcal{P} = \langle P, \preceq_\mathcal{P} \rangle$ and $\mathcal{Q} = \langle Q, \preceq_\mathcal{Q} \rangle$, their \textit{product} $\mathcal{P} \times \mathcal{Q} = \langle P \times Q, \preceq_{\mathcal{P} \times \mathcal{Q}} \rangle$ is a \glsxtrshort{acr:poset} with:
$\tupII{x_P}{x_Q} \preceq_{\mathcal{P} \times \mathcal{Q}} \tupII{y_P}{y_Q} \Leftrightarrow (x_P \preceq_\mathcal{P} y_P) \land (x_Q \preceq_\mathcal{Q} y_Q)$.
%
%
\end{definition}

\begin{definition}[Antichain in a \glsxtrshort{acr:poset}]
An \textit{antichain} is a subset $S$ of a \gls{acr:poset} $\mathcal{P}$ where no two distinct elements in $S$ are comparable:
$(x,y \in S) \land (x \preceq_\mathcal{P} y) \Rightarrow x = y$.
%
%
We denote the set of antichains of a \poset $\mathcal{P}$ by $\anti \mathcal{P}$.
\end{definition}

\begin{definition}[Upper closure]
Given a \gls{acr:poset} $\mathcal{P}$, the \textit{upper closure} of a subset $S \subseteq \mathcal{P}$ is a subset $\upper S \subseteq \mathcal{P}$ that contains all elements $y$ of $\mathcal{P}$ that are greater or equal to some element $x$ in $S$:
$\upper S = \set{ y \in P \mid \exists x \in S : x \preceq_\mathcal{P} y}$
%
%
\end{definition}

\begin{definition}[Monotone map]
A map $f \colon \mathcal{P} \to \mathcal{Q}$ between \posets \tupII{P}{\preceq_\mathcal{P}} and \tupII{Q}{\preceq_\mathcal{Q}} is \emph{monotone} if $x \preceq_\mathcal{P} y$ implies $f(x) \preceq_\mathcal{Q} f(y)$. Monotonicity is preserved by composition and products.
\end{definition}

For a subset~$S\subseteq P$ of a \poset $\mathcal{P}$, we write
$\min_{\preceq_\mathcal{P}} S \coloneq \set{ x \in S \mid \nexists y \in S : y \prec_\mathcal{P} x }$
for the set of minimal elements of~$S$, and
$\max_{\preceq_\mathcal{P}} S \coloneq \set{ x \in S \mid \nexists y \in S : x \prec_\mathcal{P} y }$
for the set of maximal elements. 
Whenever nonempty, both~$\min_{\preceq_\mathcal{P}} S$ and~$\max_{\preceq_\mathcal{P}} S$ are antichains.
When $\mathcal{P}$ is a lattice, we write $\bigvee S$ and $\bigwedge S$ for the join
(supremum) and meet (infimum) of $S$, respectively, whenever they exist.

The classical multi-objective optimization setting is recovered by taking the objective space to be~$\mathbb{R}^m$ equipped with the componentwise order.
In that case, weak Pareto dominance is simply the ambient partial order, and Pareto-optimal points are precisely the minimal (or maximal) elements of the feasible set~\cite{ehrgott2005multicriteria}.
The present paper works instead with arbitrary \posets, which is the natural level of generality for co-design problems in which quantities may only be partially comparable.


\subsection{Background on Co-Design Theory}

Co-design provides a compositional language for reasoning about systems with heterogeneous interfaces and coupled design decisions.
Its basic objects are \emph{\F{functionalities}} and \emph{\R{resources}}, modeled as \posets.
Informally, more functionality is preferred, whereas fewer resources are preferred.

\begin{definition}[\glsdefhere{\Glsxtrfull}{acr:dp}]
\label{def:dp}
Given \glspl{acr:poset}~$\funPosetF$ and $\resPosetR$ of \F{functionalities} and \R{resources}, a \emph{\glsxtrfull{acr:dp}} is an upper set of $\funPosetF^\op \times \resPosetR$.
We denote the set of such \glspl{acr:dp} by $\dpOf{\funPosetF}{\resPosetR}$.
Given a \gls{acr:dp} $\dprb$, a pair $\tup{\F{f}, \R{r}}$ of functionality $\F{f}$ and resource $\R{r}$ is \emph{feasible} if $\tup{\F{f}, \R{r}} \in \dprb$.
We order $\dpOf{\funPosetF}{\resPosetR}$ by inclusion: $\dprb_{a} \preceq \dprb_{b} \Leftrightarrow \dprb_{a} \subseteq \dprb_{b}$.
%
\end{definition}

\begin{definition}[\Gls{acr:dpi}]
\label{def:dpi}
Given \glspl{acr:poset}~$\funPosetF$ and $\resPosetR$ of \F{functionalities} and \R{resources}, a \emph{\glsxtrfull{acr:dpi}} is a tuple \tupIII{\impSetI}{\prov}{\req}, with a set of \I{implementations} $\impSetI$ and maps $\prov \colon \impSetI \to \funPosetF$ and $\req \colon \impSetI \to \resPosetR$.
For each design choice $\I{i} \in \impSetI$, $\provImp{i}$ represents the \F{functionality} provided by $\I{i}$, while $\reqImp{i}$ represents the \R{resource} it requires.
We denote the set of such \glspl{acr:dpi} by $\dpiOf{\impSetI}{\funPosetF}{\resPosetR}$.
\end{definition}

For each \glsxtrshort{acr:dpi}, there is a corresponding \gls{acr:dp} given by the free choice among all implementations: 
$\upper \set{ \tupII{\provImp{i}}{\reqImp{i}} \mid \I{i} \in \impSetI} \in \dpOf{\funPosetF}{\resPosetR}$.

\begin{definition}[Queries for \gls{acr:dpi}]
\label{def:fix-fun-min-res}
Suppose we have a \gls{acr:dpi}~\tupIII{\impSetI}{\prov}{\req}, whose functionality and resource \gls{acr:poset} are $\funPosetF$ and $\resPosetR$, respectively.
We define the \emph{Fix \F{functionality} minimize \R{resources}} query:
For a fixed functionality~$\F{f} \in \funPosetF$,
return implementations~$\I{i}$ that satisfy~$\F{f}$ and minimize the resources 
(i.e., $\reqImp{i} \in \min_{\preceq_\resPosetR} \set{ \reqImp{j} \mid \F{f} \preceq_\funPosetF \provImp{j}, \I{j} \in \impSetI} $.
We denote the set of minimal resources for this query by $\FixFunMinResOf{f} \in \anti \resPosetR$.
\end{definition}

\input{Figures/design_problem}

A \gls{acr:dp} is a generalization of a black-box multi-objective problem~\cite{ehrgott2005multicriteria,censi2022}, where one only models the space of design choices $\impSetI$ and a mapping $\req$ from design choices to objectives. In \glspl{acr:dp}, we also model what satisfying the requirements provides through the functionality space $\funPosetF$ and the map $\prov$, which allows us to compose a \gls{acr:dp} with other \glspl{acr:dp} to construct larger problems in a modular way~(\Cref{fig:dp_composition}).

\begin{example}[Observer--sensor design]
\label{ex:observer_sensor}
Consider the following open-loop \gls{acr:lti} discrete-time system:
\begin{align}
    x_{t+1} &= Ax_t + w_t, \\
    y_t &= Cx_t + v_t,
\end{align}
with $x_t \in \mathbb{R}^n$, $y_t \in \mathbb{R}^p$, and Gaussian noises $w_t \sim \mathcal{N}(0, W)$ and $v_t \sim \mathcal{N}(0, V)$ with covariances $W$ and $V$, respectively.
A system designer jointly selects sensing hardware and a state-estimation algorithm \cite{zardini2021co,zardini2022task}.
Specifically, the \textbf{Sensor} hardware is modeled by a \gls{acr:dp} $\dprb_{\mathrm{sen}}$ that requires \R{cost} ($\mathbb{R}_{\geq 0}$) and \R{power} ($\mathbb{R}_{\geq 0}$), and provides \F{noise-covariance} $V$ ($\mathbb{S}_+^p$).
The covariance matrices are ordered by the reverse Loewner order; that is, for $V, V' \in \mathbb{S}_+^p$,
\begin{align}
    V' \preceq V
    \iff
    x^\top V x \leq x^\top V' x \quad \forall x \in \mathbb{R}^p.
\end{align}
Thus, the monotonicity of $\dprb_{\mathrm{sen}}$ captures the intuition that smaller covariance is preferred, although achieving it may require more resources.
The \textbf{Observer} design is modeled by a \gls{acr:dp} $\dprb_{\mathrm{obs}}$ that requires an upper bound on \R{noise-covariance} $\overline{V}$ ($\mathbb{S}_+^p$) and provides a certified \F{error-covariance bound} $E$ ($\mathbb{S}_+^n$) for state estimation.
Connecting $\dprb_{\mathrm{sen}}$ and $\dprb_{\mathrm{obs}}$ in series expresses that, when sufficiently informative sensing is available, a desired estimation performance can be guaranteed.
\end{example}


\subsection{Performance Metrics}
\label{sec:performance_metrics}
To evaluate the quality of an approximate antichain produced by an online algorithm, we need a way to quantify how well a candidate set of non-dominated trade-offs approximates the true optimal set.
We separate two levels of structure. 
At the weakest level, only the partial order on $\resPosetR$ is available; at a richer level, one may also require topology, geometry, or measure on $\resPosetR$.

Let $\OptAnti \coloneq \FixFunMinResOf{f}$ denote the true minimal antichain for a fixed target functionality, and let $\SubAnti\subseteq \resPosetR$ be an approximate (i.e., suboptimal) antichain.
For an arbitrary \poset~$\resPosetR$, the canonical comparison is via dominated regions: we say that~$A$ \emph{dominates}~$B$ if~$\upper A \supseteq \upper B$.
Accordingly, exact recovery at the level of sets means~$\SubAnti=\OptAnti$, whereas exact recovery at the level of dominated regions means~$\upper \SubAnti = \upper \OptAnti$.

Richer structure on $\resPosetR$ enables richer metrics.
When~$\resPosetR = \mathbb{R}^m$ with the componentwise order, a Lebesgue measure is available and one may use quantitative indicators.
Given a fixed reference point~$\resref \in \mathbb{R}^m$, the \emph{hypervolume indicator} of an antichain~$A\subseteq \mathbb{R}^m$ is
\begin{align}
    \HVOf{A} = \lambda^m \Big( \big\{ \R{r} \in \mathbb{R}^m \mid \exists \R{a} \in A: \R{a} \preceq_\resPosetR \R{r} \preceq_\resPosetR \resref \big\} \Big),
    \label{eq:hv}
\end{align}
where $\lambda^m$ is the $m$-dimensional Lebesgue measure.
This quantity measures the size of the region (i.e., the Lebesgue volume) dominated by $A$ inside the box defined by the reference point~\cite{li2019quality}.
The hypervolume indicator is a standard unary quality indicator known to be strictly Pareto-compliant: if one antichain dominates another, it has strictly larger hypervolume~\cite{zitzler2003performance}.
We adopt the \emph{hypervolume difference} as a concrete performance metric over $\mathbb{R}^m$:
\begin{align}
    \HVDOf{\OptAnti}{\SubAnti}
    = 
    \HVOf{\OptAnti} - \HVOf{\SubAnti} \geq 0.
    \label{eq:hvd}
\end{align}
The hypervolume difference is zero if and only if $\OptAnti$ dominates the same region as $\SubAnti$ (up to measure zero), and it increases as the approximate front~$\SubAnti$ becomes worse. 
There are other quality metrics, such as the generational distance, inverted generational distance, and epsilon-indicator, each capturing different aspects of approximation quality (see~\cite{li2019quality} for a comprehensive survey).

We emphasize that the algorithmic framework developed in the subsequent sections is \emph{agnostic to the specific choice of metric}.
The order-theoretic machinery only requires the partial orders on $\funPosetF$ and $\resPosetR$.
Quantitative indicators such as~\eqref{eq:hvd} enter only in numerical experiments that explicitly assume Euclidean structure.

%% file: Figures/design_problem.tex
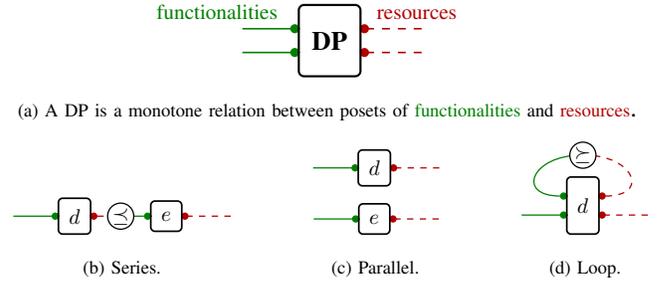
\begin{figure}[tb]
\centering

\begin{minipage}{\columnwidth}
    \centering
    \makebox[\linewidth][c]{%
    \hspace*{-5.5mm}%
    \begin{tikzpicture}[DP]
        \node[dp={2}{2}] (cnt) {DP};
        \draw[runconn, runame={resources}, relres=above, posres=0.9] (cnt_res1){};
        \draw[runconn, runame={}, relres=above, posres=0.9] (cnt_res2){};
        \draw[funconn, funame={functionalities}, relfun=above, posfun=1.5] (cnt_fun1){};
        \draw[funconn, funame={}, relfun=above, posfun=1.15] (cnt_fun2){};
    \end{tikzpicture}
    }

    \smallskip

    {\scriptsize (a) A \gls{acr:dp} is a monotone relation between
    \posets of \F{functionalities} and \R{resources}}.
    \label{fig:mathcodesign}
\end{minipage}

\vspace{8pt}

\begin{minipage}[b]{0.37\columnwidth}
    \centering
    \scalebox{0.8}{
    \begin{tikzpicture}[DP]
        \node[dp={1}{1}] (f) {$d$};
        \node[dp={1}{1}, right=1cm of f] (g) {$e$};
        \draw[rconn, rcname={}, fcname={}] (f_res1) to (g_fun1);
        \draw[runconn, runame={}] (g_res1);
        \draw[funconn, funame={}] (f_fun1);
    \end{tikzpicture}
    }

    \smallskip

    {\scriptsize (b) Series.}
\end{minipage}
\hfill
\begin{minipage}[b]{0.23\columnwidth}
    \centering
    \scalebox{0.8}{
    \begin{tikzpicture}[DP]
        \node[dp={1}{1}] (f) {$d$};
        \node[dp={1}{1}, below=0.3cm of f] (g) {$e$};
        \draw[runconn, runame={}] (f_res1){};
        \draw[runconn, runame={}] (g_res1){};
        \draw[funconn, funame={}] (f_fun1){};
        \draw[funconn, funame={}] (g_fun1){};
    \end{tikzpicture}
    }

    \smallskip

    {\scriptsize (c) Parallel.}
\end{minipage}
\hfill
\begin{minipage}[b]{0.24\columnwidth}
    \centering
    \scalebox{0.8}{
    \begin{tikzpicture}[DP]
        \node[dp={2}{2}] (f) {$d$};
        \draw[runconn, runame={}] (f_res2){};
        \draw[funconn, funame={}] (f_fun2){};
        \draw[rconn,rcname={},fcname={},feedback=1,loos=3] (f_res1) -- ($(f)+(0,4)$) |- (f_fun1);
    \end{tikzpicture}
    }

    \smallskip

    {\scriptsize (d) Loop.}
\end{minipage}

\caption{\Glspl{acr:dp} can be composed in different ways.}
\label{fig:dp_composition}

\end{figure}

%% file: Sections/online_design.tex
\section{Online Design Problem}
\label{sec:online_design}

In this section, we describe an online sequential decision-making process for multi-objective design problems.

\subsection{Problem Setup}
\label{sec:problem_setup}

We consider an agent (i.e., a decision maker) that interacts with a \gls{acr:dpi}~$\dprb \in \dpiOf{\impSetI}{\funPosetF}{\resPosetR}$ in an online and adaptive manner.
The agent can query $\dprb$ and adjust its decisions based on the observed outcomes.
Because the problem involves \posets $\funPosetF$ and $\resPosetR$, the agent aims to identify a \emph{set} of non-dominated solutions to $\dprb$. 
We make the following assumptions about $\dprb$.
\begin{assumption}[Top and bottom elements]
\label{asm:top_bottom}
The \poset $\funPosetF$ has a top element $\top_\funPosetF$ and a bottom element $\bot_\funPosetF$. Similarly, the \poset $\resPosetR$ has a top element $\top_\resPosetR$ and a bottom element $\bot_\resPosetR$.
\end{assumption}
\begin{assumption}[Feasibility and upward generation by minima]
\label{asm:feasibility}
Given any $\F{f} \in \funPosetF$, there exists a nonempty subset $\impSetI_{\F{f}} \subseteq \impSetI$ such that $\F{f} \preceq_\funPosetF \provImp{i}$ for all $\I{i} \in \impSetI_{\F{f}}$.
Moreover, the set of minimal elements $\FixFunMinResOf{f}$ exists, and every target-feasible resource is above one of these minimal elements.
\end{assumption}
\Cref{asm:top_bottom} is not restrictive, in the sense that we can always augment the \poset $\funPosetF$ and $\resPosetR$ with artificial top and bottom elements. 
We will use the top and bottom elements as a convenient way to define default behavior for algorithms.
\Cref{asm:feasibility} implies that the agent does not need to reason about returning a certificate of infeasibility for any given target functionality $\F{f}$, and the object of interest is well-posed%
\footnote{For the existence of and upward generation by minima, it is sufficient that all \posets are \glspl{acr:dcpo} and $\dprb$ is Scott continuous~\cite{zardini2023co}.}.
Further, we make the following assumption about the implementation space~$\impSetI$.
\begin{assumption}[Sampling measure]
\label{asm:base_measure}
The implementation space $\impSetI$ is a Polish space (complete, separable, metrizable) equipped with a Borel probability measure $\mu$ of full support, i.e., 
$\mu(U) > 0$ for every non-empty open set $U \subseteq \impSetI$.
\end{assumption}
\Cref{asm:base_measure} guarantees that a naive agent drawing i.i.d.\ samples from $\mu$ will eventually 
explore every region of $\impSetI$: no open subset is invisible to the sampling process. 
This subsumes the two standard cases: finite $\impSetI$ with any strictly positive distribution, and a bounded subset of $\mathbb{R}^d$ with normalized Lebesgue measure.
However, this does not impose any additional structure on~$\impSetI$.
In these special cases, it is natural to take $\mu$ to be the uniform distribution (counting measure and Lebesgue measure, respectively).
Now, the online learning problem is as follows:

\begin{problem}[Online design]
\label{prb:online_design}
The agent is given a target functionality $\F{f} \in \funPosetF$ and is allowed to interact with a \gls{acr:dpi} $\dprb \in \dpiOf{\impSetI}{\funPosetF}{\resPosetR}$ by specifying $\I{i} \in \impSetI$ and observing $\provImp{i}$ and $\reqImp{i}$. The agent is given a total evaluation budget of $N$ for querying $\dprb$. The goal of the agent is to return a set $I_N \subseteq \impSetI$, $ \card{I_N} \leq N$, that induces a resource antichain that recovers the true solution $\FixFunMinResOf{f}$.
\end{problem}

We formalize the notion of induced antichains in the next subsection. 
Note that we do not require the agent to return every optimal implementation; if two implementations share the same nondominated resource value, returning either one is sufficient. 
Finally, a variant of \Cref{prb:online_design} can be formulated in terms of a time budget rather than a sample budget.

\subsection{Histories and Induced Antichains}
\label{sec:histories}

In this section, we define useful abstractions that would help us reason about algorithms for \Cref{prb:online_design}.
At any point in time $t \in [1, N]$, we define the history $H_t$ of the agent up to that point as:
\begin{align}
    H_t = \set{ \tupIII{\I{i}}{\provImp{i}}{\reqImp{i}} \mid \I{i} \in I_t },
\end{align}
where $I_t \subseteq \impSetI$ is \textit{all} implementations selected by the agent so far. 
Moreover, the set of all \textit{finite} valid histories $\hist$ is:
\begin{align}
    \hist = \mathsf{Finite}\big( \set{ \tupIII{\I{i}}{\provImp{i}}{\reqImp{i}} \mid \I{i} \in \impSetI } \big),
\end{align}
which is a proper subset of all finite triples of $\impSetI \times \funPosetF \times \resPosetR$. 
We define a projection operator $\Pi_\resPosetR \colon \hist \to \powOf{\resPosetR}$ that extracts the resources from a given history $H \in \hist$:
\begin{align}
    \Pi_\resPosetR(H) = \set{ \R{r} \mid \exists \I{i},\F{f} : \tupIII{\I{i}}{\F{f}}{\R{r}} \in H }.
\end{align}
We define $\Pi_\impSetI \colon \hist \to \powOf{\impSetI}$ and $\Pi_\funPosetF \colon \hist \to \powOf{\funPosetF}$ in the same way. 
Furthermore, we also define projection operators $\Pi_{\impSetI\times\funPosetF} \colon \hist \to \mathsf{Pow}(\impSetI \times \funPosetF)$, $\Pi_{\impSetI\times\resPosetR} \colon \hist \to \mathsf{Pow}(\impSetI \times \resPosetR)$, and $\Pi_{\funPosetF\times\resPosetR} \colon \hist \to \mathsf{Pow}(\funPosetF \times \resPosetR)$ that extract implementation-functionality pairs, implementation-resource pairs, and functionality-resource pairs from the history.

The history $H_t$ will induce an antichain representing the set of implementations that are non-dominated by any other implementation in $H_t$.
We define a map $\mathsf{Anti}_\mathcal{P} \colon \hist \to \mathsf{Pow}(\mathcal{P})$ that extracts an antichain in a \poset $\mathcal{P}$ from a history $H \in \hist$:
\begin{align}
    \antiOf{\mathcal{P}}(H) = \set{ x \in \Pi_{\mathcal{P}}(H) \mid \nexists y \in \Pi_{\mathcal{P}}(H) : y \prec_{\mathcal{P}} x },
\end{align}
where $y \prec_{\mathcal{P}} x$ denotes that $y$ dominates $x$ under the partial order $\preceq_{\mathcal{P}}$. 
We extract the antichain of resources through $\mathsf{Anti}_\resPosetR$ and the antichain of functionalities through $\mathsf{Anti}_{\funPosetF^\text{op}}$. 
We now define a map that extracts an antichain in $\resPosetR$ that satisfies the target functionality $\F{f}$ from the history. 
%
\begin{definition}[History-induced antichain]
Given a history $H \in \hist$, the target-feasible induced resource antichain $\antiOf{\F{f}} \colon \hist \to \powOf{\resPosetR}$ is 
\begin{align}
    \antiOf{\F{f}}(H) = \mathsf{Anti}_\resPosetR(\set{ \tupIII{\I{i}}{\F{f'}}{\R{r}} \in H \mid \F{f} \preceq_\funPosetF \F{f'} }).
\end{align}
\end{definition}

\subsection{Objective}

After $N$ queries, the agent returns a set $I_N$ that induces the approximate antichain $\SubAnti = \antiOf{\F{f}}(H_N)$. Recall that $\OptAnti \coloneq \FixFunMinResOf{f}$. Exact recovery occurs when $\SubAnti = \OptAnti$. 
Note that, in the general partially ordered setting, the requirement that every point on the true antichain be covered within a specified tolerance would require extra structure on~$\resPosetR$.
For instance, in the Euclidean case $\resPosetR = \mathbb{R}^m$, $\HVD$ can be used as a quantitative measure of convergence (i.e., the agent aims to drive $\mathsf{HVD} \to 0$ using as few queries $ N$).

%% file: Sections/optimistic_elimination.tex
\section{Optimistic Evaluators and Elimination-Based Algorithm}
In this section, we develop a generic elimination-based algorithm for the online design problem described in \Cref{prb:online_design}. 
The main idea is to equip the agent with history-dependent bounds on the unknown maps~$\req$ and~$\prov$.
These bounds make it possible to certify that certain unqueried implementations cannot contribute to the target-feasible nondominated set, and may therefore be discarded without evaluation.
We first introduce optimistic and pessimistic evaluators, then present a rejection-sampling algorithm based on optimistic elimination, and finally describe several structural settings in which such evaluators can be constructed explicitly.

\subsection{Optimistic Evaluators}
\label{sec:optimistic_eval}

We begin by formalizing lower and upper bounds on the resource and functionality maps.

\begin{definition}[Optimistic evaluators]
\label{def:opt_evaluators}
Maps $\provOpt \colon \impSetI \times \hist \to \funPosetF$ and $\reqOpt \colon \impSetI \times \hist \to \resPosetR$ are called \emph{optimistic evaluators} if, for every implementation~$\I{i} \in \impSetI$ and every history~$H \in \hist$,
\begin{align}
    \reqOpt(\I{i}, H) &\preceq_\resPosetR \reqImp{i}, 
    \label{eq:optimistic_req}\\
    \provImp{i} &\preceq_\funPosetF \provOpt(\I{i}, H), 
    \label{eq:optimistic_prov}
\end{align}
\end{definition}

Thus, $\reqOpt$ is a valid lower bound on the required resources~$\req$ and $\provOpt$ is a valid upper bound on the provided functionality $\prov$ under any history. 
The term \emph{optimistic} reflects the fact that resources are to be minimized and functionalities are to be maximized, so that~$\reqOpt(\I{i},H)$ and~$\provOpt(\I{i},H)$ represent the most favorable assessment of implementation $\I{i}$ consistent with the information available in the history $H$.
Note that because of \Cref{asm:top_bottom}, optimistic evaluators always exist in a trivial sense:
\begin{equation*}
\reqOpt(\I{i},H) \equiv \bot_\resPosetR,
\qquad
\provOpt(\I{i},H) \equiv \top_\funPosetF.
\end{equation*}
Such bounds are valid but uninformative.
The purpose of the history dependence is precisely to allow the bounds to tighten as information accumulates.
Typically, one expects~$\reqOpt(\I{i},H)$ to increase in the resource order and~$\provOpt(\I{i},H)$ to decrease in the functionality order as~$H$ grows, thereby yielding progressively sharper elimination rules.
We will present precise instances of $\reqOpt$ and $\provOpt$ in \Cref{sec:optimistic_evaluator_example}.

We now explain how optimistic evaluators can be used to discard implementations without querying them.
Fix a target functionality~$\F{f} \in \funPosetF$, a history~$H \in \hist$, and an implementation~$\I{i} \in \impSetI \setminus \Pi_\impSetI(H)$ (i.e., an implementation that the agent did not query yet). 
Suppose that either of the following conditions holds:
\begin{align}
    \reqOpt(\I{i}, H) &\in \upper \antiOf{\F{f}}(H),
    \label{eq:optimistic_req_eliminate}\\
    \provOpt(\I{i}, H) &\prec_\funPosetF \F{f}.
    \label{eq:optimistic_prov_eliminate}
\end{align}
Condition \eqref{eq:optimistic_req_eliminate} states that even the most favorable resource estimate for~$\I{i}$ is already dominated by the current target-feasible antichain, whereas \eqref{eq:optimistic_prov_eliminate} states that even the most favorable functionality estimate fails to satisfy the target requirement.
In either case, $\I{i}$ cannot contribute a new element to the solution of the query associated with $\F{f}$.
Indeed, by \eqref{eq:optimistic_req} and \eqref{eq:optimistic_prov}, these conditions imply
\begin{equation*}
\reqImp{i} \in \upper \antiOf{\F{f}}(H)
\qquad\text{or}\qquad
\F{f} \npreceq_\funPosetF \provImp{i},
\end{equation*}
respectively.
Hence, $\I{i}$ may be safely discarded without explicitly observing either $\reqImp{i}$ or $\provImp{i}$. 
This is the basic elimination principle underlying the algorithm below. 
A graphical illustration is shown in \Cref{fig:elimination_example}.

\input{Figures/elimination_mixed_example}

For completeness, we also introduce the dual notion of pessimistic evaluation.
\begin{definition}[Pessimistic evaluators]
\label{def:pes_evaluators}
Maps $\provPes \colon \impSetI \times \hist \to \funPosetF$ and $\reqPes \colon \impSetI \times \hist \to \resPosetR$ are called \emph{pessimistic evaluators} if, for every implementation~$\I{i} \in \impSetI$ and every history~$H \in \hist$,
\begin{align}
    \reqImp{i} &\preceq_\resPosetR \reqPes(\I{i}, H),
    \label{eq:pessimistic_req}\\
    \provPes(\I{i}, H) &\preceq_\funPosetF \provImp{i}.
    \label{eq:pessimistic_prov}
\end{align}
\end{definition}
In contrast to optimistic evaluators, pessimistic evaluators provide certified worst-case guarantees. 
In particular, if for some implementation $\I{i}$ one has
\begin{equation*}
    \reqPes(\I{i},H) \notin \upper \antiOf{\F{f}}(H)
\qquad\text{and}\qquad
\F{f} \preceq_\funPosetF \provPes(\I{i},H),
\end{equation*}
then $\I{i}$ is guaranteed to satisfy the target and to improve the current feasible antichain. 
Although the baseline algorithm below only requires optimistic evaluators, pessimistic evaluators are useful for designing more aggressive acquisition and stopping rules.

\subsection{Rejection Sampling with Optimistic Evaluations}
\label{sec:rejection_sampling}

We now describe a simple elimination-based procedure that uses optimistic evaluators to focus sampling on implementations that may still improve the current solution estimate.

\begin{definition}[Admissible set]
\label{def:admissible_set}
Fix a target functionality $\F{f} \in \funPosetF$.
For a history $H \in \hist$, the \emph{admissible set} of implementations $\mathcal{A}(H) \subseteq \impSetI$ is
\begin{align*}
\mathcal{A}(H)
\coloneq
\Big\{
\I{i} \in \impSetI \setminus \Pi_\impSetI(H)
\ \Big|\
&\reqOpt(\I{i},H) \notin \upper \antiOf{\F{f}}(H),\\
&\F{f} \preceq_\funPosetF \provOpt(\I{i},H)
\Big\}. \numberthis
\label{eq:admissible_set}
\end{align*}
\end{definition}

Thus, $\mathcal{A}(H)$ consists precisely of those unqueried implementations that cannot yet be ruled out as irrelevant to the target query. 
Equivalently, it is the set of implementations that remain plausible candidates for improving the current target-feasible antichain.

\begin{algorithm}[tb]
\caption{Rejection Sampler with Optimistic Evaluators}
\label{alg:elimination_sampler}
\begin{algorithmic}[1]
    \State \textbf{input:} Target functionality $\F{f}$ and evaluation budget $N$
    \State \textbf{initialize:} History $H= \{\}$
    \For{$t = 1, \dots, N$}
        \State Set $\mathsf{Accept} = \bot$
        \While{$\neg \mathsf{Accept}$}
            \State Sample $\I{i} \sim \mu(\impSetI \setminus \Pi_\impSetI(H))$
            \State Evaluate $\R{r}^\opt = \reqOpt(\I{i}, H)$ and $\F{f}^\opt = \provOpt(\I{i}, H)$
            \State Set $\mathsf{Accept} = \top$ if $( \R{r}^\opt \notin \upper \antiOf{\F{f}}(H) ) \land ( \F{f} \preceq_\funPosetF \F{f}^\opt )$
        \EndWhile

        \State \textbf{query expensive block:}
        \State Query $\R{r} = \reqImp{i}$ and $\F{f'} = \provImp{i}$ from $\dprb$
        \State Append $\tupIII{\I{i}}{\F{f'}}{\R{r}}$ to the history $H$
    \EndFor
    \State \textbf{return:} Antichain $\mathsf{Anti}_{\F{f}}(H)$ and its corresponding implementations
\end{algorithmic}
\end{algorithm}

\Cref{alg:elimination_sampler} may be viewed as a rejection sampler over the current admissible set. 
At each round, the algorithm draws implementations from the base measure $\mu$ until it finds one that is not eliminated by the optimistic tests, and only then queries the true values of $\req$ and $\prov$. 
In this way, the algorithm concentrates its budget on the subset of implementations that remain potentially relevant under the information encoded in the history.
To ensure the rejection-sampling loop is well defined up to budget $N$, we assume that the admissible set is always nonempty and has positive $\mu$-measure along histories of interest.

\begin{assumption}[Well-defined rejection sampling]
\label{asm:well_defined_rejection}
For every history $H \in \hist$, the admissible set $\mathcal{A}(H)$ is $\mu$-measurable with $\mu(\mathcal{A}(H)) > 0$.
\end{assumption}

Now, we list formal guarantees on our algorithm. 
The next lemma formalizes the soundness of the elimination rule. 
\begin{lemma}[Safe elimination by optimistic evaluators]
\label{lem:safe_elimination}
Let $H \in \hist$ and $\I{i} \in \impSetI$. 
Assume either $\reqOpt(\I{i},H) \in \upper \antiOf{\F{f}}(H)$ or $\F{f} \npreceq_{\funPosetF} \provOpt(\I{i},H)$. Then:
\begin{align}
    \neg\Big(
    \F{f} \preceq_{\funPosetF} \provImp{i}
    \ \land\
    \reqImp{i} \notin \upper \antiOf{\F{f}}(H)
    \Big).
\end{align}
Equivalently, $\I{i}$ cannot be both target-feasible and an improvement over the current feasible antichain.
\end{lemma}
\begin{proof}
We prove each case separately.

\emph{Case 1}: Suppose that~$\reqOpt(\I{i},H) \in\upper \antiOf{\F{f}}(H)$.
By definition of upper closure, there exists some~$\R{r}_H \in \antiOf{\F{f}}(H)$ such that~$\R{r}_H \preceq_\resPosetR \reqOpt(\I{i},H)$. 
By optimism \eqref{eq:optimistic_req},~$\reqOpt(\I{i},H) \preceq_\resPosetR \reqImp{i}$. 
By transitivity,~$\R{r}_H \preceq_\resPosetR \reqImp{i}$, hence $\reqImp{i} \in\upper \antiOf{\F{f}}(H)$.

\emph{Case 2}: Suppose that~$\F{f} \npreceq_{\funPosetF} \provOpt(\I{i},H)$.
By optimism \eqref{eq:optimistic_prov}, $\provImp{i} \preceq_\funPosetF \provOpt(\I{i},H)$. If $\F{f} \preceq_\funPosetF \provImp{i}$, then by transitivity $\F{f} \preceq_\funPosetF \provOpt(\I{i},H)$, a contradiction. Hence $\F{f} \npreceq_{\funPosetF} \provImp{i}$.
\end{proof}

The next result shows that, under a natural monotonicity condition, the admissible set can only shrink as more information is collected.

\begin{lemma}[Monotone shrinking of the admissible set]
\label{lem:monotone_shrinking}
Let~$\I{i} \in \impSetI$ and~$H, H' \in \hist$ such that~$H \subseteq H'$. 
Assume that:
\begin{align*}
    \reqOpt(\I{i}, H) &\preceq_\resPosetR \reqOpt(\I{i}, H'), \\
    \provOpt(\I{i}, H') &\preceq_\funPosetF \provOpt(\I{i}, H). \numberthis
    \label{eq:monotone_in_history}
\end{align*}
Then $\mathcal{A}(H') \subseteq \mathcal{A}(H)$.
\end{lemma}
\begin{proof}
Since $H \subseteq H'$, we have $\upper \antiOf{\F{f}}(H) \subseteq \upper \antiOf{\F{f}}(H')$. Let $i \in \mathcal{A}(H')$. Then $\reqOpt(\I{i}, H') \notin \upper \antiOf{\F{f}}(H')$, which implies that $\reqOpt(\I{i}, H') \notin \upper \antiOf{\F{f}}(H)$. Now, by hypothesis \eqref{eq:monotone_in_history}, $\reqOpt(\I{i}, H) \preceq_\resPosetR \reqOpt(\I{i}, H')$. If $\reqOpt(\I{i}, H) \in \upper \antiOf{\F{f}}(H)$, then since the upper closure is an upper set and $\reqOpt(\I{i}, H) \preceq_\resPosetR \reqOpt(\I{i}, H')$, we would have $\reqOpt(\I{i}, H') \in \upper \antiOf{\F{f}}(H)$, contradicting what we just established. Hence $\reqOpt(\I{i}, H) \notin \upper \antiOf{\F{f}}(H)$.

Similarly, $\F{f} \preceq_\funPosetF \provOpt(\I{i}, H')$ and $\provOpt(\I{i}, H') \preceq_\funPosetF \provOpt(\I{i}, H)$ yield $\F{f} \preceq_\funPosetF \provOpt(\I{i}, H)$ by transitivity. Therefore $i \in \mathcal{A}(H)$.
\end{proof}

Together, \Cref{lem:safe_elimination} and \Cref{lem:monotone_shrinking} establish the two central properties of \Cref{alg:elimination_sampler}: \emph{soundness} (no implementation that could still improve the current target-feasible frontier is ever eliminated) and \emph{monotone tightening} (the set of candidate implementations can only shrink as data accumulates).
Now, we use these properties to establish antichain recovery results for \Cref{alg:elimination_sampler}.

\begin{theorem}[Trichotomy for optimal witnesses]
\label{thm:optimal_witnesses}
Consider \Cref{alg:elimination_sampler} without a sample budget under
\Cref{asm:top_bottom,asm:feasibility,asm:base_measure,asm:well_defined_rejection},
with valid optimistic evaluators (\Cref{def:opt_evaluators}).
Let $H_t$ denote the history after round~$t$.
Then for every $\R{r}^\star \in \OptAnti$ and every witness $\I{i}^\star$ satisfying
$\F{f} \preceq_\funPosetF \provImpStar{i}$ and $\reqImpStar{i} = \R{r}^\star$,
almost surely exactly one of the following occurs:
\begin{enumerate}[label=(\alph*)]
\item \emph{(Queried.)} $\I{i}^\star$ is queried at some round~$t$, in which case $\R{r}^\star \in \upper \antiOf{\F{f}}(H_t)$.
\item \emph{(Eliminated.)} $\I{i}^\star$ is never queried but is eliminated from $\mathcal{A}(H_t)$ at some round~$t$, in which case $\R{r}^\star \in \upper \antiOf{\F{f}}(H_t)$.
\item \emph{(Persistently admissible.)} $\I{i}^\star$ is never queried and remains admissible at every round, i.e., $\I{i}^\star \in \mathcal{A}(H_t)$ for all $t \geq 1$.
\end{enumerate}
In Cases~(a) and~(b), $\R{r}^\star$ is covered by the approximate antichain.
\end{theorem}
\begin{proof}
Fix $\R{r}^\star \in \OptAnti$ and a witness $\I{i}^\star$.
The three cases are mutually exclusive and exhaustive.

\emph{Case (a):}
Suppose $\I{i}^\star \in \Pi_\impSetI(H_t)$ at some round~$t$.
Since $\I{i}^\star$ is target-feasible ($\F{f} \preceq_\funPosetF \provImpStar{i}$), its resource $\R{r}^\star = \reqImpStar{i}$ enters $\antiOf{\F{f}}(H_t)$, giving $\R{r}^\star \in \upper \antiOf{\F{f}}(H_t)$.

\emph{Case (b):}
Suppose $\I{i}^\star \notin \Pi_\impSetI(H_t)$ for all $t$ (never queried), but $\I{i}^\star \notin \mathcal{A}(H_t)$ at some round~$t$ (eliminated).
By \Cref{lem:safe_elimination}, $\I{i}^\star$ cannot be both target-feasible and an improvement over the current antichain.
Since $\I{i}^\star$ is target-feasible by assumption, it follows that $\reqImpStar{i} \in \upper \antiOf{\F{f}}(H_t)$, i.e., $\R{r}^\star \in \upper \antiOf{\F{f}}(H_t)$ through another witness.

\emph{Case (c):}
Suppose $\I{i}^\star \notin \Pi_\impSetI(H_t)$ and $\I{i}^\star \in \mathcal{A}(H_t)$ for all $t$.
This is precisely the persistent-admissibility alternative.
\end{proof}

\begin{theorem}[Exact recovery from positive-measure witness sets]
\label{thm:positive_measure_convergence}
Under the assumptions of \Cref{thm:optimal_witnesses}, additionally assume that $\OptAnti$ is finite and that for every $\R{r}^\star \in \OptAnti$, the set
\begin{align}
    W(\R{r}^\star)
    \coloneq
    \set{
    \I{i} \in \impSetI :
    \F{f} \preceq_\funPosetF \provImp{i},
    \ \reqImp{i} = \R{r}^\star
    }
\end{align}
is Borel-measurable and satisfies $\mu(W(\R{r}^\star)) > 0$.
Then there exists a random $N_0 < \infty$ such that
\begin{align}
    \upper \antiOf{\F{f}}(H_t) = \upper \OptAnti,
    \qquad \forall t \geq N_0,
    \quad \text{almost surely.}
\end{align}
\end{theorem}
\begin{proof}
Fix $\R{r}^\star \in \OptAnti$ and let $W^\star \coloneq W(\R{r}^\star)$.
We show that $\R{r}^\star$ is covered in finite time almost surely.

If there exists $\I{i}^\star \in W^\star$ falling under Case~(a) or~(b) of
\Cref{thm:optimal_witnesses}, then there exists a finite round~$t$ such that $\R{r}^\star \in \upper \antiOf{\F{f}}(H_t)$.
It remains to rule out the event that $\R{r}^\star$ is never covered.

Suppose $\R{r}^\star \notin \upper \antiOf{\F{f}}(H_t)$ for all $t$.
Then no witness $\I{i}^\star \in W^\star$ can fall under Case~(a) or~(b) of \Cref{thm:optimal_witnesses}, since either outcome would cover $\R{r}^\star$.
Hence every $\I{i}^\star \in W^\star$ falls under Case~(c), so $W^\star \subseteq \mathcal{A}(H_t)$ for all $t$.

At every round $t$, the conditional probability of sampling some witness in $W^\star$ is
\begin{align}
\label{eq:prob_optimal_witness}
    \Pr(\I{i}_t \in W^\star \mid H_{t-1})
    =
    \frac{\mu(W^\star)}{\mu(\mathcal{A}(H_{t-1}))}
    \geq
    \mu(W^\star)
    > 0,
\end{align}
where we used $W^\star \subseteq \mathcal{A}(H_{t-1}) \subseteq \impSetI$ and $\mu(\impSetI)=1$ (\Cref{asm:base_measure}).
On the event that $\R{r}^\star$ is never covered, this bound holds at every round, so $\sum_{t \geq 1} \Pr(\I{i}_t \in W^\star \mid H_{t-1}) = \infty$.
By the conditional Borel--Cantelli lemma (L\'evy), $\I{i}_t \in W^\star$ for some finite time almost surely, contradicting the assumption for Case~(c) that $\R{r}^\star$ is never covered.

Therefore, for each $\R{r}^\star \in \OptAnti$, only Cases~(a) and~(b) occur, so $\R{r}^\star \in \upper \antiOf{\F{f}}(H_t)$ at some finite random time.
Since the queried set grows monotonically with~$t$, the sequence $\{\upper \antiOf{\F{f}}(H_t)\}_{t \geq 1}$ is nondecreasing, so there exists a finite random time $T(\R{r}^\star)$ such that $\R{r}^\star \in \upper \antiOf{\F{f}}(H_t)$, $\forall t \geq T(\R{r}^\star)$, almost surely.
Since $\OptAnti$ is finite, the random variable $N_0 \coloneq \max_{\R{r}^\star \in \OptAnti} T(\R{r}^\star)$ is almost surely finite, giving $\upper \OptAnti \subseteq \upper \antiOf{\F{f}}(H_N)$, $\forall\, N \geq N_0$, almost surely.
The reverse inclusion holds for all $N \geq 1$: every target-feasible queried implementation has a resource vector in $\upper \OptAnti$, so $\upper \antiOf{\F{f}}(H_N) \subseteq \upper \OptAnti$.
Combining the two inclusions proves the claim.
\end{proof}

\begin{remark}[Role of the optimistic evaluators]
\label{rem:convergence_rate}
Beyond correctness (i.e., case~(b) of~\Cref{thm:optimal_witnesses}), the value of non-trivial optimistic evaluators shows in~\eqref{eq:prob_optimal_witness}, where the probability of sampling an optimal witness $\I{i}^\star$ increases as the set $\mathcal{A}(H_{t-1})$ shrinks. While tightening is guaranteed by~\Cref{lem:monotone_shrinking}, the statement is not quantitative, and quantitative statements would require extra assumptions on the partial orders $\funPosetF$ and $\resPosetR$.
\end{remark}

The positive-measure assumption in \Cref{thm:positive_measure_convergence} is needed because if an optimal witness set has measure zero, it can stay admissible forever but still never be sampled, so exact recovery is no longer guaranteed.
Finally, \Cref{alg:elimination_sampler} is intentionally a baseline procedure.
When $\impSetI$ is finite, the admissible set can often be maintained explicitly.
In structured implementation spaces, one may replace pure rejection sampling by more efficient search mechanisms, such as maintaining an explicit outer approximation of~$\mathcal{A}(H)$, using branch-and-bound over a hierarchical partition of~$\impSetI$, or combining optimistic elimination with priority rules derived from pessimistic evaluators. 
All such refinements preserve the same elimination principle.
A key advantage of \Cref{alg:elimination_sampler} is that it is extendable to a more general setting, as we will show in \Cref{sec:compositional}.

\subsection{Structural Instantiations of Optimism}
\label{sec:optimistic_evaluator_example}

In this section, we discuss several structural settings in which optimistic evaluators can be constructed explicitly. 
For clarity, we focus on the resource evaluator $\reqOpt$; corresponding upper-bound constructions for $\provOpt$ are obtained by dual arguments.
In the absence of structural information, one may always initialize $\reqOpt(\cdot,\emptyset)=\bot_\resPosetR$. 
The examples below yield sharper bounds whenever additional structure is available.

\subsubsection{Monotonicity}
\label{sec:monotonicity}

Assume that $\impSetI$ is itself a \poset and that $\req \colon \impSetI \to \resPosetR$ is monotone, i.e.,
$\I{j} \preceq_{\impSetI} \I{i} \Rightarrow \reqImp{j} \preceq_{\resPosetR} \reqImp{i}, \forall \I{i}, \I{j} \in \impSetI$.
%
%
Given a history $H \in \hist$ and implementation $\I{i} \in \impSetI$, define
\begin{align}
\label{eq:monotonicity_opt_set}
    S(\I{i},H) := \set{ \R{r} \mid (\tupII{\I{j}}{\R{r}} \in \Pi_{\impSetI\times\resPosetR}(H)) \land ( \I{j} \preceq_{\impSetI} \I{i} ) }.
\end{align}
If $S(\I{i},H)=\emptyset$, set $\reqOpt(\I{i},H)=\bot_\resPosetR$. 
Otherwise, any maximal element of $S(\I{i},H)$ may be used as an optimistic evaluator:
$\reqOpt(\I{i}, H) \in \max_{\preceq_{\resPosetR}} S(\I{i},H)$.
%
%
In general, $\max_{\preceq_{\resPosetR}} S$ is an antichain (with respect to the ordering $\succeq_\resPosetR$) and always exists (because $S$ is finite).
If, in addition, $\resPosetR$ is a join-semilattice, a canonical choice is
\begin{align}
\label{eq:monotone-opt}
    \reqOpt(\I{i},H) = \bigvee S(\I{i},H),
\end{align}
which is the tightest lower bound induced by the queried predecessors of $\I{i}$.

\begin{lemma}
\Cref{eq:monotone-opt} defines a valid optimistic evaluator (\Cref{def:opt_evaluators}).
\end{lemma}
\begin{proof}
For every queried implementation $\I{j}$ satisfying $\I{j} \preceq_\impSetI \I{i}$, monotonicity gives~$\reqImp{j} \preceq_\resPosetR \reqImp{i}$.
Hence, every element of $S(\I{i},H)$ is a lower bound on $\reqImp{i}$. 
Any maximal element of this set is therefore still a valid lower bound, which proves optimism. 
If $\bigvee S(\I{i},H)$ exists, then it is the least upper bound of all elements in $S(\I{i},H)$, and since $\reqImp{i}$ is itself an upper bound of $S(\I{i},H)$, we obtain~$\bigvee S(\I{i},H) \preceq_\resPosetR \reqImp{i}.$
\end{proof}

\begin{remark}
A technical advantage of the join-based construction is that it is automatically monotone in history: enlarging the history can only add elements to $S(\I{i},H)$ and therefore can only increase $\reqOpt(\I{i},H)$. 
By contrast, if one selects an arbitrary maximal element of $S(\I{i},H)$, the monotonicity property in \eqref{eq:monotone_in_history} need not hold without an additional tie-breaking rule.
\end{remark}

\subsubsection{Lipschitz Continuity}
\label{sec:lipschitzness}

Assume that $\resPosetR \subseteq \mathbb{R}^m$ is equipped with the componentwise order, that $\impSetI$ is a metric space with metric $d$, and that $\req$ is $L$-Lipschitz with respect to $d$ and the $\ell_p$ norm:
$\lVert \reqImp{i} - \reqImp{j} \rVert_{p} \leq L d(\I{i},\I{j}), \forall \I{i},\I{j} \in \impSetI$.
%
%
Given a history $H \in \hist$ and implementation $\I{i} \in \impSetI$, define
$S(\I{i},H) \coloneq \set{ \R{r} - L d(\I{i},\I{j}) \mathbf{1} \mid \tupII{\I{j}}{\R{r}} \in \Pi_{\impSetI\times\resPosetR}(H) }$.
%
%
If $S(\I{i},H)=\emptyset$, set $\reqOpt(\I{i},H)=\bot_\resPosetR$. 
Otherwise, define
\begin{align}
\label{eq:lipschitz-opt}
    \reqOpt(\I{i}, H) = \bigvee S(\I{i},H).
\end{align}
%
If desired, this lower bound may additionally be clipped below by $\bot_\resPosetR$. 
This construction is the same ``optimism via smoothness'' idea used in Lipschitz bandits~\cite{bubeck2011x, lattimore2020bandit}.

\begin{lemma}
\Cref{eq:lipschitz-opt} defines a valid optimistic evaluator (\Cref{def:opt_evaluators}).
\end{lemma}
\begin{proof}
For any coordinate $k \in [m]$ and any queried~$\I{j}$, we have $|\req_k(\I{i}) - \req_k(\I{j})| \leq \|\reqImp{i} - \reqImp{j}\|_\infty \leq \|\reqImp{i} - \reqImp{j}\|_p \leq L d(\I{i},\I{j})$. Hence $\req_k(\I{j}) - L d(\I{i},\I{j}) \leq \req_k(\I{i})$, which gives the componentwise lower bound. Taking the componentwise maximum over all queried points yields the tightest such bound.
\end{proof}

\begin{remark}
This construction is automatically monotone in history: as $H$ grows, the set of candidate lower bounds only enlarges, and therefore $\reqOpt(\I{i},H)$ can only increase in the resource order.
\end{remark}

\subsubsection{Linear-Parametric Models}
\label{sec:linear_parametric}

As a third example, assume that $\resPosetR \subseteq \mathbb{R}^m$ is equipped with the componentwise order and that the resource map $\req$ belongs to a known linear-parametric class. 
Specifically, suppose there exist a known feature map~$\Phi \colon \impSetI \to \mathbb{R}^{m \times p}$ and an unknown parameter vector~$\theta^\star \in \mathbb{R}^p$ such that:
$\reqImp{i} = \Phi(\I{i})\, \theta^\star, \quad \forall \I{i} \in \impSetI$.
%
%
Further assume that the agent is given a prior set~$\Theta_0 \subseteq \mathbb{R}^p$ such that~$\theta^\star \in \Theta_0$. 
For a history $H \in \hist$, define the corresponding \emph{confidence set} as the subset of $\Theta_0$ consistent with all observations:
\begin{align}
    \Theta(H) = \set{ \theta \in \Theta_0 \mid \Phi(\I{j})\,\theta = \R{r} \,\, \forall \tupII{\I{j}}{\R{r}} \in \Pi_{\impSetI \times \resPosetR}(H) }.
    \label{eq:confidence_set}
\end{align}
By construction, $\theta^\star \in \Theta(H)$ for every $H \in \hist$, and $\Theta(H') \subseteq \Theta(H)$ whenever $H \subseteq H'$.
Given a history $H \in \hist$ and an implementation $\I{i} \in \impSetI$, define
$S(\I{i},H) \coloneq \set{ \Phi(\I{i})\,\theta \mid \theta \in \Theta(H) }$.
%
%
Assuming that the coordinatewise infimum of $S(\I{i},H)$ belongs to $\resPosetR$, define
\begin{align}
\label{eq:linear-opt}
    \reqOpt(\I{i}, H) = \bigwedge S(\I{i}, H).
\end{align}
Equivalently, each coordinate is obtained by solving a separate optimization:
\begin{align}
\label{eq:linear-opt-cord}
    [\reqOpt(\I{i}, H)]_k = \min_{\theta \in \Theta(H)} \, [\Phi(\I{i})\,\theta]_k, \qquad \forall k \in [m].
\end{align}
When $\Theta_0$ is a convex polytope, the confidence set $\Theta(H)$ is also a convex polytope (as the intersection of $\Theta_0$ with affine constraints), and each coordinate in \eqref{eq:linear-opt-cord} reduces to a \gls{acr:lp}.

\begin{lemma}
\Cref{eq:linear-opt} defines a valid optimistic evaluator (\Cref{def:opt_evaluators}).
\end{lemma}
\begin{proof}
Since $\theta^\star \in \Theta(H)$ for all $H \in \hist$, the true value $\reqImp{i} = \Phi(\I{i})\,\theta^\star$ belongs to the set $S$. Hence, for each coordinate $k \in [m]$:
\begin{align*}
    [\reqOpt(\I{i}, H)]_k 
    = 
    \min_{\theta \in \Theta(H)} [\Phi(\I{i})\,\theta]_k 
    \leq 
    [\Phi(\I{i})\,\theta^\star]_k = \req_k(\I{i}),
\end{align*}
which gives $\reqOpt(\I{i}, H) \leq_{\resPosetR} \reqImp{i}$ componentwise. Moreover, since $\Theta(H') \subseteq \Theta(H)$ whenever $H \subseteq H'$, we have:
\begin{align*}
    \reqOpt(\I{i}, H) 
    &= 
    \bigwedge \set{ \Phi(\I{i})\,\theta \mid \theta \in \Theta(H) } \\
    &\leq_\resPosetR 
    \bigwedge \set{ \Phi(\I{i})\,\theta \mid \theta \in \Theta(H') } 
    = 
    \reqOpt(\I{i}, H'),
\end{align*}
confirming that the bound tightens monotonically as data accumulates~\eqref{eq:monotone_in_history}.
\end{proof}

\begin{remark}[Exact recovery in the linear-parametric case]
\label{rem:linear_recovery}
Stacking the observation constraints in \eqref{eq:confidence_set} yields $\Theta(H) = \{\theta \in \Theta_0 \mid \mathbf{\Phi}_H\,\theta = \mathbf{r}_H\}$, where $\mathbf{\Phi}_H \in \mathbb{R}^{nm \times p}$ concatenates the feature matrices of the $n$ queried implementations.
For exact recovery $\Theta(H) = \{\theta^\star\}$, it is sufficient that $\operatorname{rank}(\mathbf{\Phi}_H) = p$.
Consequently, if one can query implementations whose feature blocks contribute independent rank, then at most $p$ rank-one queries are sufficient.
At that point $\reqOpt(\I{i}, H) = \reqImp{i}$ for all $\I{i} \in \impSetI$ simultaneously, and the admissible set $\mathcal{A}(H)$ coincides with the one induced by the true maps.
\end{remark}

\begin{remark}[Model misspecification]
The analysis in this section is deterministic: it assumes that \eqref{eq:optimistic_req} and \eqref{eq:optimistic_prov} hold for all implementations and all histories. 
If the optimistic evaluators are instead given by high-probability confidence bounds, then the same soundness guarantees hold on the corresponding confidence event. 
If model misspecification is a concern, one may combine elimination with forced exploration, for example, by querying from the base measure $\mu$ with a small probability $\delta>0$ and using the rejection sampler otherwise. 
This prevents a single erroneous optimistic certificate from permanently excluding an implementation.
\end{remark}

%% file: Figures/elimination_mixed_example.tex
\begin{figure}[tb]
\centering
\hfill
\begin{minipage}[c]{0.20\columnwidth}
\centering
\setlength{\abovecaptionskip}{0pt}
\raisebox{0.19in}{\includegraphics[height=1.81in]{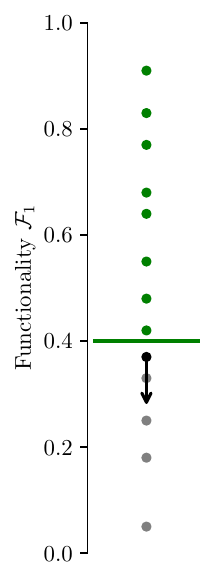}}
\end{minipage}
\begin{minipage}[c]{0.65\columnwidth}
\centering
\begin{tikzpicture}
\tikzset{edge/.style={draw=gray, thick}}
    \matrix (A) [matrix of nodes, row sep=0.5cm, column sep=0.3cm]
    {%
    & & & \color{gray}{l}\\
    & \color{gray}{i} & & \color{gray}{j} & & \color{gray}{k} \\
    \textbf{\R{e}} & & \textbf{\R{f}} & & \textbf{g} & & \color{gray}{h}\\
    & \color{gray}{b} & & \color{gray}{c} & & \textbf{\R{d}} \\
    & & & {\color{gray}a}\\
    };
    \draw[edge] (A-5-4)--(A-4-2);
    \draw[edge] (A-5-4)--(A-4-4);
    \draw[edge] (A-5-4)--(A-4-6);
    \draw[edge] (A-4-2)--(A-3-1);
    \draw[edge] (A-4-2)--(A-3-3);
    \draw[edge] (A-4-2)--(A-3-5);
    \draw[edge] (A-4-4)--(A-3-3);
    \draw[edge] (A-4-4)--(A-3-7);
    \draw[edge] (A-4-6)--(A-3-5);
    \draw[edge] (A-4-6)--(A-3-7);
    \draw[edge] (A-3-1)--(A-2-2);
    \draw[edge] (A-3-1)--(A-2-4);
    \draw[edge] (A-3-3)--(A-2-2);
    \draw[edge] (A-3-3)--(A-2-6);
    \draw[edge] (A-3-5)--(A-2-4);
    \draw[edge] (A-3-5)--(A-2-6);
    \draw[edge] (A-3-7)--(A-2-6);
    \draw[edge] (A-2-2)--(A-1-4);
    \draw[edge] (A-2-4)--(A-1-4);
    \draw[edge] (A-2-6)--(A-1-4);
\end{tikzpicture}
\vspace{12pt}
\end{minipage}
\hfill
\caption{
Example for an implementation $\I{i}$ that can be eliminated using either of the optimistic evaluators. Left: The functionality space ($\funPosetF=\mathbb{R}$) with $\F{f} = 0.4$. The green points satisfy the functionality, while the gray points do not. The black point with an arrow shows the value of $\provOpt(\I{i}, H)$. Right: Hasse diagram of the resource space $\resPosetR$ (discrete lattice) with the current antichain highlighted in red. The black bold letter shows the value of $\reqOpt(\I{i}, H)$.
}
\label{fig:elimination_example}
\end{figure}

%% file: Sections/online_codesign.tex
\section{Compositional Co-Design and Optimistic Propagation}
\label{sec:compositional}

We now study the case where the online problem described in \Cref{prb:online_design} adheres to a compositional structure. 
This setting is important for two related reasons. 
First, the system-level implementation space may be far too large to explore directly, even when each constituent block is relatively simple. 
Second, in many applications, only a small subset of the blocks are expensive to evaluate, while the remaining blocks admit explicit models and can therefore be handled algorithmically. 
The resulting problem is to learn an expensive subsystem \emph{through} the surrounding co-design graph, rather than in isolation.

\subsection{Multigraph Model}
\label{sec:co_design}

We model the compositional structure by a directed multigraph whose nodes are atomic \glspl{acr:dpi} and whose edges encode interconnections between resource and functionality coordinates~\cite{zardini2023co}.

\begin{definition}[Co-design problem]
\label{def:cdp}
A \emph{co-design problem} $\cdprb$ is a directed multigraph $\tupII{\mathcal{V}}{\mathcal{E}}$, where:
\begin{itemize}
    \item $\mathcal{V} = \set{v_1, \dots, v_K}$ is a set of atomic \glspl{acr:dpi}, each $v \in \mathcal{V}$ being a tuple $\tupIII{\impSetI_v}{\prov_v}{\req_v}$ over \posets $\funPosetF_v$ and $\resPosetR_v$ (\Cref{def:dpi}).
    \item $\mathcal{E}$ is a set of directed edges. An edge $e \in \mathcal{E}$ is a tuple $\tupII{ \tupII{v_1}{i_1} }{ \tupII{v_2}{i_2} }$, where $v_1, v_2 \in \mathcal{V}$ are two nodes, $i_1$ is an index of a resource of $v_1$, and $i_2$ is an index of a functionality of $v_2$, such that $\pi_{i_1} \resPosetR_{v_1} = \pi_{i_2} \funPosetF_{v_2}$%
    \footnote{$\pi_{i}$ indicates the projection operator that selects the $i$-th component of the corresponding tuple or product poset.}.
\end{itemize}
Every $\cdprb$ induces an equivalent system-level \gls{acr:dpi}, constructed as follows:
\begin{itemize}
    \item the system-level functionality space $\funPosetF$ is the product of all functionality coordinates not internally connected by edges in~$\mathcal{E}$;
    \item the system-level resource space $\resPosetR$ is the product of all resource coordinates not internally connected by edges in~$\mathcal{E}$;
    \item the system-level implementation space $\impSetI$ is the subset of $\prod_{v \in \mathcal{V}} \impSetI_v$ consisting of all tuples $\I{i}$ satisfying every interconnection constraint, for all~$\tupII{\tupII{v_1}{i_1}}{\tupII{v_2}{i_2}} \in \mathcal{E}$:
    \begin{align}
        \pi_{i_1}\req_{v_1}\big(\pi_{v_1}\I{i}\big)
        \preceq
        \pi_{i_2}\prov_{v_2}\big(\pi_{v_2}\I{i}\big);
        \label{eq:interconnection_constraint}
    \end{align}
    \item the system-level maps $\prov$ and $\req$ are obtained by projecting $\I{i} \in \impSetI$ onto the unconnected functionality and resource coordinates, respectively.
\end{itemize}
\end{definition}

\Cref{def:cdp} should be viewed as a structured representation of a large design problem. 
The atomic blocks specify local trade-offs, while the edges enforce compatibility among them (\Cref{fig:co-design}). 
The induced system-level \gls{acr:dpi} is therefore not an arbitrary Cartesian product of local implementation spaces: it is the subset of tuples consistent with the internal resource--functionality wiring. 
This is precisely the structure that enables local optimistic certificates to be propagated to system-level ones.

\input{Figures/co-design}

\subsection{Online Co-Design Problem}
\label{sec:online_codesign}

We now consider the case in which only one atomic block is expensive to evaluate.
Assume that a single node~$v_q \in \mathcal{V}$ is expensive, while all other nodes in~$\mathcal{V}\setminus\set{v_q}$ are computationally tractable. 
The online agent therefore chooses implementations only for~$v_q$, while the remaining blocks are handled by a solver for the tractable portion of the diagram.

The appropriate notion is that of a \emph{completion}: once the agent fixes a local implementation for the expensive block, the rest of the graph must be completed so that all interconnection constraints are satisfied and the external functionality target is met.

\begin{definition}[Completion]
\label{def:completion}
Given a co-design problem~$\cdprb = \tupII{\mathcal{V}}{\mathcal{E}}$ (\Cref{def:cdp}), an expensive block~$v_q \in \mathcal{V}$, a local implementation~$\I{i}_q \in \impSetI_{v_q}$, and a target functionality~$\F{f} \in \funPosetF$, a \emph{completion} of $\I{i}_q$ for $\F{f}$ is an element~$\I{j} \in \prod_{v \in \mathcal{V} \setminus \set{v_q}} \impSetI_v$ such that the composite implementation~$\I{i} \coloneq \impPointIqJ$ satisfies:
\begin{enumerate}
    \item \textbf{Interconnection feasibility}: All internal wires are satisfied. That is, $\forall \tupII{ \tupII{v_1}{i_1} }{ \tupII{v_2}{i_2} } \in \mathcal{E}$:
    \begin{align}
        \pi_{i_1} \req_{v_1}( \pi_{v_1} \I{i} ) \preceq \pi_{i_2} \prov_{v_2}( \pi_{v_2} \I{i} ).
    \end{align}
    \item \textbf{Target feasibility}: The system-level functionality meets the target:
    \begin{align}
        \F{f} \preceq_{\funPosetF} \prov(\I{i}).
    \end{align}
\end{enumerate}
We denote the set of all completions of $\I{i}_q$ for $\F{f}$ by $J_{\F{f}}(\I{i}_q)$%
\footnote{Under \gls{acr:dcpo} interface \posets and Scott-continuous atomic \glspl{acr:dp}, any nonempty completion problem has a minimal antichain representation $A$ whose upward closure is the full feasible completion-resource set.}.
\end{definition}

A local implementation of the expensive block can induce several system-level trade-offs, depending on how the tractable part of the graph is completed. 
The relevant object is therefore not a single resource value, but an antichain of nondominated system-level resources.

\begin{definition}[System-level antichain induced by a local implementation]
\label{def:local_to_system_antichain}
Given $\I{i}_q \in \impSetI_{v_q}$ and a target functionality $\F{f} \in \funPosetF$, define the \emph{induced system-level antichain} as:
\begin{align}
    \solveOf{\F{f}}(\I{i}_q)
    \coloneq
    \min_{\preceq_{\resPosetR}}\Big\{ \req(\I{i}) \ \Big|\ \I{j} \in J_{\F{f}}(\I{i}_q) \Big\}.
    \label{eq:system_antichain}
\end{align}
Thus, $\solveOf{\F{f}}(\I{i}_q) \in \anti \resPosetR \cup \set{\emptyset}$.
%
\end{definition}

\begin{remark}
When $\solveOf{\F{f}}(\I{i}_q)=\emptyset$, the local implementation $\I{i}_q$ cannot be completed to satisfy the target. 
Otherwise, $\solveOf{\F{f}}(\I{i}_q)$ is the system-level nondominated resource antichain that remains achievable after fixing $\I{i}_q$.
\end{remark}

To update the history, the environment need not return \emph{all} completions of $\I{i}_q$. 
It is enough to return one witness completion for each element of the antichain in \eqref{eq:system_antichain}.

\begin{definition}[Witness set]
\label{def:witness_set}
For $\I{i}_q \in \impSetI_{v_q}$, let $W_{\F{f}}(\I{i}_q)$ be the \emph{witness set}, i.e., any subset of $J_{\F{f}}(\I{i}_q)$ such that
\begin{align}
\set{ \req\big(\impPointIqJ\big) \mid \I{j} \in W_{\F{f}}(\I{i}_q) }
=
\solveOf{\F{f}}(\I{i}_q).
\end{align}
That is, $W_{\F{f}}(\I{i}_q)$ contains at least one completion witnessing each element of the system-level antichain.
\end{definition}

We can now state the online learning problem in the compositional setting.

\begin{problem}[Online co-design]
\label{prb:online_codesign}
Given a co-design problem $\cdprb = \tupII{\mathcal{V}}{\mathcal{E}}$ (\Cref{def:cdp}) with an expensive block $v_q \in \mathcal{V}$ and tractable blocks $\mathcal{V} \setminus \set{v_q}$, the agent is given a target functionality $\F{f} \in \funPosetF$ and interacts with the system as follows:
\begin{enumerate}
    \item the agent selects a local implementation $\I{i}_q \in \impSetI_{v_q}$.
    \item the environment solves the tractable completion problem and returns the system-level antichain $\solveOf{\F{f}}(\I{i}_q)$ together with any witness set $W_{\F{f}}(\I{i}_q)$;
    \item the agent may issue at most $N$ such queries for~$v_q$ (budget).
\end{enumerate}
The goal is to return a set~$I_N \subseteq \impSetI_{v_q}$,~$\card{I_N}\le N$, whose induced system-level resource antichain recovers the true solution $\FixFunMinResOf{f}$.
\end{problem}

\Cref{prb:online_codesign} makes explicit that a single expensive query may induce several composite implementations at the system level. 
Conceptually, evaluating a candidate $\I{i}_q$ means \emph{solving the rest of the diagram}: deciding whether $J_{\F{f}}(\I{i}_q)$ is nonempty and, if so, computing the nondominated system-level resources that remain achievable.
We visualize this setup in \Cref{fig:co-design}.

\begin{example}
Building on~\Cref{ex:observer_sensor}, the \gls{acr:dp} $\dprb_{\mathrm{sen}}$ may require hardware design and validation, and can therefore be regarded as the expensive block~$v_q$. 
By contrast, once a family of state estimation algorithms is fixed, the \gls{acr:dp} $\dprb_{\mathrm{obs}}$ may be tractable.
In this case, $\I{i}_q$ is a sensor choice, and the completion solves for target-compatible observer designs.
\end{example}

\begin{remark}[Tractability of completions]
Consider $\dprb_1 \in \dpiOf{\impSetI_1}{\funPosetF}{\mathcal{\R{Q}}}$ and $\dprb_2 \in \dpiOf{\impSetI_2}{\mathcal{\F{Q}}}{\resPosetR}$ with finite $\impSetI_1$ and $\impSetI_2$, and their series interconnection $\cdprb \in \dpiOf{\impSetI_1 \times \impSetI_2}{\funPosetF}{\resPosetR}$.
Assume that $\dprb_1$ is the expensive block and $\dprb_2$ is the tractable block.
The solution map $h \colon \mathcal{\F{Q}} \to \anti \resPosetR$ for the tractable block,
\begin{align}
    h(\F{q}) \coloneqq \min_{\preceq_{\resPosetR}} \set{ \req_2(\I{i}_2) \mid \F{q} \preceq_{\mathcal{\F{Q}}} \prov_2(\I{i}_2),\, \I{i}_2 \in \impSetI_2 },
\end{align}
maps each functionality $\F{q} \in \mathcal{\F{Q}}$ to its minimal resources and can be computed offline without knowing $\dprb_1$.
Then, for every online query $\I{i}_1 \in \impSetI_1$, the system-level functionality is $\prov_1(\I{i}_1)$ and the system-level minimal resources are $h(\req_1(\I{i}_1))$.
A complexity statement of this kind can be established for general $\cdprb$ with an arbitrary multigraph~\cite{zardini2023co}.
On the other hand, if one treats $\cdprb$ in this example as a monolithic black-box optimization problem, one might need to enumerate $\card{\impSetI_1} \times \card{\impSetI_2}$ implementations in the worst case.
\end{remark}

\subsection{Propagating the Optimistic Evaluation}
\label{sec:propagating_optimistic_eval}

We now show how local optimistic information on the expensive block propagates through the co-design graph to produce a system-level optimistic certificate.

Recall that the agent is given a target system functionality $\F{f} \in \funPosetF$, where $\funPosetF$ is the product of the unconnected functionalities of the multigraph. 
We assume local optimistic bounds $\reqOpt_q(\cdot, H)$ and $\provOpt_q(\cdot, H)$ for $v_q$, as described in \Cref{sec:optimistic_eval}. 
Because the remaining atomic problems $\mathcal{V} \setminus \set{v_q}$ are computationally tractable and fully known, we set their optimistic evaluators to their true maps (i.e., $\reqOpt_v = \req_v$ and $\provOpt_v = \prov_v$ for all $v \neq q$).
For a composite implementation~$\I{i} \in \prod_{v\in\mathcal V}\impSetI_v$, let~$\reqOpt(\I{i},H)$ and~$\provOpt(\I{i},H)$ denote the corresponding system-level optimistic functionality/resource values obtained by projecting the unconnected wires after replacing each map by its optimistic version.

\begin{definition}[Optimistically valid completions]
\label{def:optimistic_completion}
Given a history $H \in \hist$ and a local implementation $\I{i}_q \in \impSetI_{v_q}$, the set of \emph{optimistically valid completions} is
\begin{align}
    &J_{\F{f}}^\opt(\I{i}_q, H) 
    = \nonumber \\
    &\Bigg\{ 
    \I{j} \in \prod_{v \neq q} \impSetI_v 
    \Biggm| 
    \Big( 
    \pi_{i_1} \reqOpt_{v_1}( \pi_{v_1} \I{i}, H ) 
    \preceq 
    \pi_{i_2} \provOpt_{v_2}( \pi_{v_2} \I{i}, H ) \nonumber \\
    &\hspace{8pt} \forall \tupII{\tupII{v_1}{i_1}}{\tupII{v_2}{i_2}} \in \mathcal{E}
    \Big)
    \land \left( \F{f} \preceq_{\funPosetF} \provOpt(\I{i}, H) \right) 
    \Bigg\},
    \label{eq:opt_completions}
\end{align}
where $\I{i}\coloneq \impPointIqJ$.
\end{definition}

\begin{remark}
This definition is obtained from \Cref{def:completion} by replacing the true maps on the expensive block by optimistic ones. 
Because optimism relaxes feasibility constraints and improves apparent performance,~$J_{\F{f}}^\opt(\I{i}_q,H)$ is a superset of the true completion set.
\end{remark}

\begin{definition}[Propagated optimistic antichain]
\label{def:propagated_optimistic_antichain}
Given $\I{i}_q \in \impSetI_{v_q}$ and $H \in \hist$, define the \emph{propagated optimistic system-level antichain} by
\begin{align}
    \solveOf{\F{f}}^\opt(\I{i}_q, H)
    \coloneq
    \min_{\preceq_{\resPosetR}}\Big\{ \reqOpt(\I{i}, H) \ \Big|\ \I{j} \in J_{\F{f}}^\opt(\I{i}_q, H) \Big\},
    \label{eq:system_opt_antichain}
\end{align}
where $\reqOpt$ projects the composite implementation to the unconnected resource wires using the respective optimistic evaluators.
\end{definition}

The next result is the compositional counterpart of optimism for a single block.

\begin{theorem}[Propagation of optimistic feasibility and resources]
\label{lem:propagated_optimistic_antichain}
Fix a history $H \in \hist$ and a local implementation $\I{i}_q \in \impSetI_{v_q}$. 
Let $J_{\F{f}}(\I{i}_q)$ be the set of completions (\Cref{def:completion}), and let $J_{\F{f}}^\opt(\I{i}_q,H)$ and $\solveOf{\F{f}}^\opt(\I{i}_q,H)$ be defined as above. Assume that the local optimistic evaluators on $v_q$ satisfy \eqref{eq:optimistic_req} and \eqref{eq:optimistic_prov}. Then:
\begin{enumerate}
    \item \textbf{Optimistic feasibility is a relaxation}:
    \begin{align}
        J_{\F{f}}(\I{i}_q) \subseteq J_{\F{f}}^\opt(\I{i}_q,H).
        \label{eq:completion_relaxation}
    \end{align}
    In particular, if $J_{\F{f}}^\opt(\I{i}_q,H)=\emptyset$, then $J_{\F{f}}(\I{i}_q)=\emptyset$.
    
    \item \textbf{Lower approximation of system-level resources}:  
    If $J_{\F{f}}^\opt(\I{i}_q,H)\neq\emptyset$, then
    \begin{align}
        \upper \solveOf{\F{f}}(\I{i}_q)
        \subseteq
        \upper \solveOf{\F{f}}^\opt(\I{i}_q, H).
        \label{eq:propagated_upper_inclusion}
    \end{align}
\end{enumerate}
\end{theorem}
\begin{proof}
For every block $v \neq v_q$, the optimistic maps agree with the true maps by definition. 
Hence, the pointwise bounds \eqref{eq:optimistic_req} and \eqref{eq:optimistic_prov} hold for all blocks, with equality away from $v_q$.

We first prove \eqref{eq:completion_relaxation}.
Let $\I{j} \in J_{\F{f}}(\I{i}_q)$ and set $\I{i} \coloneq \impPointIqJ$. 
For any edge $\tupII{\tupII{v_1}{i_1}}{\tupII{v_2}{i_2}} \in \mathcal{E}$, optimism~\eqref{eq:optimistic_req} on $v_1$ and~\eqref{eq:optimistic_prov} on $v_2$ give:
\begin{align*}
    \pi_{i_1} \reqOpt_{v_1}( \pi_{v_1} \I{i}, H )
    &\preceq
    \pi_{i_1} \req_{v_1}( \pi_{v_1} \I{i} ) \\
    &\preceq
    \pi_{i_2} \prov_{v_2}( \pi_{v_2} \I{i} ) \\
    &\preceq
    \pi_{i_2} \provOpt_{v_2}( \pi_{v_2} \I{i}, H )
\end{align*}
where the middle inequality is the true interconnection constraint from \Cref{def:completion}. 
Similarly, $\F{f} \preceq_{\funPosetF} \prov(\I{i}) \preceq_{\funPosetF} \provOpt(\I{i}, H)$ by target feasibility and~\eqref{eq:optimistic_prov}. 
Hence $\I{j} \in J_{\F{f}}^\opt(\I{i}_q, H)$, and since $\I{j}$ was arbitrary, $J_{\F{f}}(\I{i}_q) \subseteq J_{\F{f}}^\opt(\I{i}_q, H)$.
The contrapositive yields: if $J_{\F{f}}^\opt(\I{i}_q, H) = \emptyset$, then $J_{\F{f}}(\I{i}_q) = \emptyset$, and we may set $\solveOf{\F{f}}^\opt(\I{i}_q, H) = \emptyset$ for immediate rejection.

We next prove \eqref{eq:propagated_upper_inclusion}.
Assume $J_{\F{f}}^\opt(\I{i}_q, H) \neq \emptyset$ and take $\R{r} \in \upper \solveOf{\F{f}}(\I{i}_q)$. 
By definition of upper closure, there exists $\R{r}' \in \solveOf{\F{f}}(\I{i}_q)$ with $\R{r}' \preceq_{\resPosetR} \R{r}$. 
By~\eqref{eq:system_antichain}, there exists $\I{j}' \in J_{\F{f}}(\I{i}_q)$ with $\R{r}' = \req(\tupII{\I{i}_q}{\I{j}'})$. 
Let $\I{i}' \coloneq \tupII{\I{i}_q}{\I{j}'}$. 
By~\eqref{eq:completion_relaxation}, $\I{j}' \in J_{\F{f}}^\opt(\I{i}_q, H)$, so $\reqOpt(\I{i}', H)$ belongs to the set minimized in~\eqref{eq:system_opt_antichain}. 
There thus exists $\R{r}'' \in \solveOf{\F{f}}^\opt(\I{i}_q, H)$ such that:
\begin{align}
    \R{r}'' \preceq_{\resPosetR} \reqOpt(\I{i}', H) \preceq_{\resPosetR} \req(\I{i}') = \R{r}' \preceq_{\resPosetR} \R{r},
\end{align}
where the second inequality is optimism~\eqref{eq:optimistic_req}. 
Hence $\R{r} \in \upper \solveOf{\F{f}}^\opt(\I{i}_q, H)$.
\end{proof}

\Cref{lem:propagated_optimistic_antichain} yields the elimination rule actually used by the compositional algorithm.

\begin{corollary}[Safe elimination of local implementations]
\label{cor:safe_local_elimination}
Fix $H \in \hist$ and $\I{i}_q \in \impSetI_{v_q}$. If either
\begin{align*}
\solveOf{\F{f}}^\opt(\I{i}_q,H)=\emptyset
\,\ \text{or} \,\
\solveOf{\F{f}}^\opt(\I{i}_q,H) \subseteq \upper \antiOf{\F{f}}(H),
\end{align*}
then $\solveOf{\F{f}}(\I{i}_q) \subseteq \upper \antiOf{\F{f}}(H)$.
Equivalently, no completion of $\I{i}_q$ can add a new element to the current target-feasible antichain.
\end{corollary}

\begin{proof}
If $\solveOf{\F{f}}^\opt(\I{i}_q,H)=\emptyset$, the claim follows from the final statement of \Cref{lem:propagated_optimistic_antichain}. 
Otherwise,~$
\solveOf{\F{f}}^\opt(\I{i}_q,H) \subseteq \upper \antiOf{\F{f}}(H)
\Longrightarrow
\upper \solveOf{\F{f}}^\opt(\I{i}_q,H) \subseteq \upper \antiOf{\F{f}}(H),$ 
since $\upper \antiOf{\F{f}}(H)$ is an upper set. 
Combining this with \eqref{eq:propagated_upper_inclusion} gives
$\upper \solveOf{\F{f}}(\I{i}_q) \subseteq \upper \solveOf{\F{f}}^\opt(\I{i}_q,H) \subseteq \upper \antiOf{\F{f}}(H)$,
which implies the result.
\end{proof}

As in the single-block case, one also expects the propagated optimistic antichain to tighten monotonically with history.

\begin{proposition}[Monotonicity of propagated optimism]
\label{prop:propagated_monotonicity}
Let $H,H' \in \hist$ satisfy $H \subseteq H'$. 
Assume that the optimistic evaluators on the expensive block satisfy
\begin{align}
    \reqOpt_q(\I{i}_q,H) &\preceq \reqOpt_q(\I{i}_q,H'),
    \label{eq:propagated_monotone_req}
    \\
    \provOpt_q(\I{i}_q,H') &\preceq \provOpt_q(\I{i}_q,H),
    \label{eq:propagated_monotone_prov}
\end{align}
for every $\I{i}_q \in \impSetI_{v_q}$. Then, for every fixed $\I{i}_q \in \impSetI_{v_q}$,
\begin{align}
    J_{\F{f}}^\opt(\I{i}_q,H') &\subseteq J_{\F{f}}^\opt(\I{i}_q,H),
    \label{eq:completion_shrinks}
    \\
    \upper \solveOf{\F{f}}^\opt(\I{i}_q,H')
    &\subseteq
    \upper \solveOf{\F{f}}^\opt(\I{i}_q,H).
    \label{eq:opt_antichain_shrinks}
\end{align}
\end{proposition}

\begin{proof}
Fix $\I{i}_q \in \impSetI_{v_q}$.

To prove \eqref{eq:completion_shrinks}, let $\I{j}\in J_{\F{f}}^\opt(\I{i}_q,H')$ and set $\I{i}\coloneq \impPointIqJ$. 
Every optimistic interconnection inequality and the optimistic target-feasibility inequality hold under $H'$. 
Because resource optimism increases and functionality optimism decreases with history only on the expensive block, every such inequality that holds under $H'$ also holds under $H$. 
Hence $\I{j}\in J_{\F{f}}^\opt(\I{i}_q,H)$.

To prove \eqref{eq:opt_antichain_shrinks}, take any $\R{r}' \in \solveOf{\F{f}}^\opt(\I{i}_q,H')$.
Then there exists $\I{j}' \in J_{\F{f}}^\opt(\I{i}_q,H')$ such that, with $\I{i}'\coloneq \tupII{\I{i}_q}{\I{j}'}$, $\R{r}' = \reqOpt(\I{i}',H')$.
By \eqref{eq:completion_shrinks}, we also have $\I{j}' \in J_{\F{f}}^\opt(\I{i}_q,H)$. 
Moreover, $\reqOpt(\I{i}',H) \preceq \reqOpt(\I{i}',H')$
by \eqref{eq:propagated_monotone_req}. 
Since $\reqOpt(\I{i}',H)$ belongs to the set minimized in \eqref{eq:system_opt_antichain} under history $H$, there exists $\R{r} \in \solveOf{\F{f}}^\opt(\I{i}_q,H)$
such that $\R{r} \preceq \reqOpt(\I{i}',H) \preceq \reqOpt(\I{i}',H') = \R{r}'$.
Therefore $\R{r}' \in \upper \solveOf{\F{f}}^\opt(\I{i}_q,H)$, proving \eqref{eq:opt_antichain_shrinks}.
\end{proof}

\begin{algorithm}[tb]
\caption{Rejection Sampler with Propagated Optimistic Evaluators}
\label{alg:elimination_sampler_multigraph}
\begin{algorithmic}[1]
    \State \textbf{input:} Multigraph $\cdprb = \tupII{\mathcal{V}}{\mathcal{E}}$, expensive node $v_q$, target functionality $\F{f}$, evaluation budget $N$
    \State \textbf{initialize:} History $H=\{\}$
    \State \textbf{preprocess:} Reduce tractable subgraph $\mathcal{V}\setminus\{v_q\}$

    \For{$t = 1, \dots, N$}
        \State Set $\mathsf{Accept}=\bot$
        \While{$\neg \mathsf{Accept}$}
            \State Sample $\I{i}_q \sim \mu(\impSetI_{v_q} \setminus \Pi_{\impSetI_{v_q}}(H))$
            \State Evaluate $A^\opt \gets \solveOf{\F{f}}^\opt(\I{i}_q, H)$
            \State Set $\mathsf{Accept}=\top$ if $( A^\opt \neq \emptyset ) \land ( \exists \R{r} \in A^\opt : \R{r} \notin \upper \antiOf{\F{f}}(H) )$
        \EndWhile

        \State \textbf{query expensive block:}
        \State Receive $S \gets \set{ \tupIII{\impPointIqJ}{\prov(\impPointIqJ)}{\req(\impPointIqJ)} \mid \I{j} \in W_{\F{f}}(\I{i}_q) }$ from $\cdprb$

        \For{\textbf{each} $\tupIII{\I{i}}{\provImp{i}}{\reqImp{i}} \in S$}
            \State Append $\tupIII{\I{i}}{\provImp{i}}{\reqImp{i}}$ to the history $H$
        \EndFor
    \EndFor

    \State \textbf{return:} Antichain $\antiOf{\F{f}}(H)$ and its corresponding composite implementations
\end{algorithmic}
\end{algorithm}

\subsection{Algorithm}
\label{sec:compositional_algorithm}

\Cref{alg:elimination_sampler_multigraph} is the exact compositional analogue of \Cref{alg:elimination_sampler}. 
The compositional rejection sampler operates over local implementations of the expensive block, but uses the propagated optimistic antichain as its acceptance test.
The only difference is that the acceptance test is no longer based on a single optimistic resource value, but on the propagated optimistic system-level antichain $\solveOf{\F{f}}^\opt(\I{i}_q,H)$ associated with the local implementation $\I{i}_q$.
By design, \Cref{alg:elimination_sampler_multigraph} accommodates heterogeneous spaces; for instance, the local space $\resPosetR_v$ may be a cone-ordered vector space, whereas the global system space $\resPosetR$ may be equipped only with a discrete partial order.
\Cref{alg:elimination_sampler_multigraph} has correctness guarantees by virtue of \Cref{lem:propagated_optimistic_antichain}. 
Specifically, the propagated optimistic evaluators satisfy the pointwise bounds required by \Cref{cor:safe_local_elimination}, so composite implementations that improve the current solution are never eliminated. 
Moreover, if the local optimistic evaluators satisfy the monotonicity-in-history condition, then the system-level admissible set also shrinks monotonically (\Cref{prop:propagated_monotonicity}).

\subsection{Implementation Details}
\label{sec:implementation_details}

Evaluating the propagated optimistic antichain $\solveOf{\F{f}}^\opt(\I{i}_q, H)$ requires solving the tractable subgraph $\mathcal{V}\setminus\{v_q\}$ for each candidate $\I{i}_q$.
In general, this amounts to computing a least fixed point of a monotone operator over the multigraph, which can be costly when invoked repeatedly inside the online loop.
We employ two complementary techniques to mitigate this cost.

\textbf{Pre-solving and component merging.}
Before the online loop begins, we simplify the diagram while treating $v_q$ as an unknown interface.
Concretely, we (i)~collapse every subgraph that does not depend on $v_q$ into an equivalent reduced component, and (ii)~resolve internal feedback loops that are entirely contained within tractable blocks%
\footnote{Both operations are valid semantics-preserving rewrites in the compact closed symmetric monoidal category $\mathbf{DP}$ of design problems~\cite{zardini2023co}.}.
The result is a smaller diagram in which $v_q$ connects to the rest of the system through a minimal set of interface wires, so that each propagation step solves only the reduced problem.

\textbf{Early termination of fixed-point iterations.}
When the diagram contains feedback, propagation proceeds by iterating a monotone operator (i.e., using Kleene’s
algorithm~\cite{davey2002introduction}) until convergence.
For elimination, however, the exact fixed point is not always needed: it suffices to certify that the optimistic outcome is already dominated by $\upper \antiOf{\F{f}}(H)$.
We therefore compare with $\upper \antiOf{\F{f}}(H)$ throughout the iterations and terminate as soon as dominance is established, thereby converting many full fixed-point computations into short runs.

%% file: Figures/co-design.tex
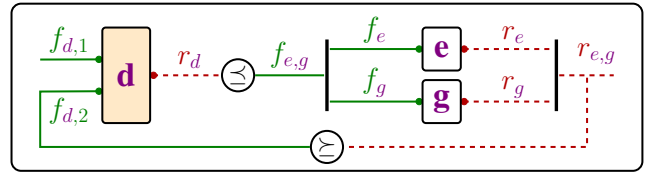
\begin{figure}[t]
\centering
\definecolor{highlight}{RGB}{255,222,172}
\begin{adjustbox}{max width=0.95\columnwidth, center}
\begin{tikzpicture}[
    DP,
    dpx=1.2cm,
    dpy=1cm,
    dp port dot scale=0.4,
    cmp/.style={
        draw,
        circle,
        fill=white,
        line width=1.2pt,
        inner sep=1.2pt,
        font=\Large
    },
]

\draw[rounded corners=2mm, line width=1.2pt] (0,0) rectangle (11,3.5);

\node[cmp] (c2) at (5.5,0.5) {$\succeq$};

\node[
    dp={2}{1},
    fill=highlight!66,
    line width=1.2pt,
    minimum height=2cm,
    minimum width=1cm,
    font=\huge,
] (d) at (2,2)
{\textcolor{dppurple}{\textbf{d}}};

\draw[dpgreen, line width=1.2pt] (0.5,2.33) -- ($(d_fun1)-(0.05,0)$);
\node[text=dpgreen, font=\LARGE] at (1,2.71)
{$f_{\textcolor{dppurple}{d},1}$};

\draw[dpgreen, line width=1.2pt]
    ($(d_fun2)-(0.05,0)$) -- (0.5,1.66) -- (0.5,0.50) -- (c2.west);
\node[text=dpgreen, font=\LARGE] at (1,1.24)
{$f_{\textcolor{dppurple}{d},2}$};

\node[cmp] (c1) at (3.95,2) {$\preceq$};

\draw[dpred, dashed, line width=1.2pt] ($(d_res1)+(0.05,0)$) -- (c1.west);
\node[text=dpred, font=\LARGE] at (3.1,2.33)
{$r_{\textcolor{dppurple}{d}}$};

\draw[dpgreen, line width=1.2pt] (c1.east) -- (5.45,2);
\node[text=dpgreen, font=\LARGE] at (4.85,2.40)
{$f_{\textcolor{dppurple}{e,g}}$};

\draw[line width=2pt] (5.5,1.25) -- (5.5,2.75);

\node[
    dp={1}{1},
    draw=black,
    fill=white,
    line width=1.2pt,
    minimum height=0.90cm,
    minimum width=0.8cm,
    font=\huge,
] (e) at (7.5,2.55)
{\textcolor{dppurple}{\textbf{e}}};

\node[
    dp={1}{1},
    draw=black,
    fill=white,
    line width=1.2pt,
    minimum height=0.90cm,
    minimum width=0.80cm,
    font=\huge,
] (g) at (7.5,1.45)
{\textcolor{dppurple}{\textbf{g}}};

\draw[dpgreen, line width=1.2pt] (5.55,2.55) -- ($(e_fun1)-(0.05,0)$);
\draw[dpgreen, line width=1.2pt] (5.55,1.45) -- ($(g_fun1)-(0.05,0)$);

\node[text=dpgreen, font=\LARGE] at (6.33,2.95)
{$f_{\textcolor{dppurple}{e}}$};
\node[text=dpgreen, font=\LARGE] at (6.33,1.85)
{$f_{\textcolor{dppurple}{g}}$};

\draw[dpred, dashed, line width=1.2pt] ($(e_res1)+(0.05,0)$) -- ($(9.5,2.55)-(0.05,0)$);
\draw[dpred, dashed, line width=1.2pt] ($(g_res1)+(0.05,0)$) -- ($(9.5,1.45)-(0.05,0)$);

\node[text=dpred, font=\LARGE] at (8.75,2.85)
{$r_{\textcolor{dppurple}{e}}$};
\node[text=dpred, font=\LARGE] at (8.75,1.75)
{$r_{\textcolor{dppurple}{g}}$};

\draw[line width=2pt] (9.5,1.25) -- (9.5,2.75);

\draw[dpred, dashed, line width=1.2pt] 
    (9.55,2.0) -- (10.5,2.0);
\draw[dpred, dashed, line width=1.2pt]
    (10.035,1.95) -- (10.035,0.50) -- (c2.east);
\node[text=dpred, font=\LARGE] at (10.2,2.45)
{$r_{\textcolor{dppurple}{e,g}}$};

\end{tikzpicture}
\end{adjustbox}
\caption{
An example of a co-design problem $\cdprb$ (\Cref{def:cdp}) with multiple atomic \glspl{acr:dpi}. The highlighted block represents a subsystem that is expensive to evaluate, e.g., because it is only available through simulation.
}
\label{fig:co-design}
\end{figure}

%% file: Sections/case_study.tex
\section{Case Study}
\label{sec:case_study}

In this section, we evaluate the proposed optimistic elimination framework on both synthetic structured problems and engineering co-design examples. 
Our goals are twofold: first, to quantify the sample-efficiency gains obtained by exploiting structural optimism; second, to assess whether those gains persist on realistic system co-design problems in which each evaluation requires solving a computationally expensive oracle.
Throughout, all methods are compared under the same budget, measured in the number of expensive evaluations.
We begin by describing the reference algorithms and the adaptations needed for a fair comparison.

\subsection{Benchmark Algorithms and Evaluation Protocol}

We compare our method against four \glspl{acr:ea}:
\MOEAD~\cite{zhang2007moea},
\NSGAIII~\cite{deb2013evolutionary},
\RVEA~\cite{cheng2016reference}, and
\KGB~\cite{ye2022knowledge}.
\MOEAD~\cite{zhang2007moea} decomposes a multi-objective problem into a set of scalar subproblems and solves them cooperatively, exploiting neighborhood structure among adjacent subproblems.
\NSGAIII~\cite{deb2013evolutionary} maintains a population through non-dominated sorting and uses a set of predefined reference points to preserve diversity across the Pareto front.
\RVEA~\cite{cheng2016reference} guides a population toward the Pareto front using reference vectors with an angle-penalized distance metric, and dynamically adapts the reference vector distribution to handle differently scaled objectives and irregular front shapes.
Finally, \KGB~\cite{ye2022knowledge} uses a naive Bayesian classifier trained on historically accumulated solutions to filter randomly generated candidates and predict a high-quality initial population for optimization.
We also compare with \qLogEHVI, a \gls{acr:bo} method with expected hypervolume improvement acquisition function whose log formulation improves numerical stability~\cite{ament2023unexpected}.

As in \Cref{prb:online_codesign}, all baselines operate on the same local implementation space~$\impSetI_{v_q}$ exposed to the online agent.
Evaluating a candidate local implementation always means invoking the same expensive oracle, namely, the environment that solves the completion problem and returns the corresponding system-level resource value (\Cref{def:completion}).
Thus, all methods are compared under a common notion of sample complexity%
\footnote{All experiments were run on a machine with an AMD Ryzen 9 7900X CPU (12 cores, 4.7\,GHz) and 64\,GB of RAM.}.

Three further adaptations are needed to make the comparisons fair and well-defined.

\textbf{Cumulative evaluation.} For a given algorithm, let~$H_t^{\mathrm{alg}} \in \hist$ denote the cumulative history consisting of all \emph{unique} expensive evaluations obtained up to budget~$t$.
The performance of the algorithm at budget~$t$ is computed from the cumulative archive
$\hat A_t^{\mathrm{alg}} \coloneq \antiOf{\F{f}}(H_t^{\mathrm{alg}})$.
Accordingly, for population-based baselines, we evaluate the antichain induced by the union of all queried candidates up to time~$t$, not merely the final population.
This makes the comparison commensurate with our framework, which explicitly accumulates information over time.
Repeated proposals of previously evaluated candidates are served from a cache and do not consume additional budget.

\textbf{Target functionality.}
When a queried implementation fails to satisfy the target functionality~$\F{f}$, it is assigned the worst resource value~$\top_{\resPosetR}$.
This converts target infeasibility into immediate domination and permits direct comparison with methods that do not natively reason about functionalities.

\textbf{Unique resource assignment.}
In the setting of \Cref{prb:online_codesign}, a single implementation may be assigned an antichain of non-dominated resources due to the diagram completion. Since the \glspl{acr:ea} we compare with expect each implementation to be assigned a \emph{single point} in the objective space, we restrict our numerical case studies to instances where every completion yields a unique resource value.

In addition to these baselines, we compare with a structure-agnostic quasi-random low-discrepancy sampler based on the \HALTON sequence~\cite{halton1964algorithm,leobacher2014introduction}.
This baseline uses the same proposal sequence as our method but performs no elimination.
Within \Cref{alg:elimination_sampler} and \Cref{alg:elimination_sampler_multigraph}, the same \HALTON sequence is used as the base sampling measure~$\mu$.

\subsection{Synthetic Problems}

We first evaluate \Cref{alg:elimination_sampler} on analytically generated problems designed to isolate the effect of structural optimism (\Cref{sec:optimistic_evaluator_example}).
In all synthetic experiments, the target functionality is fixed and satisfied by construction, so the learning problem reduces to recovering the nondominated antichain of a resource map~$g \colon [0,1]^d \to [0,1]^m$ under the componentwise order on both the implementation and resource spaces.

\subsubsection{Monotone Problems}
To generate instances that adhere to the monotone structure described in \Cref{sec:monotonicity}, we construct a piecewise-constant map~$g \colon [0,1]^d \to [0,1]^m$ by summing~$K$ monotone step atoms per output coordinate.
Specifically, we draw thresholds~$t_k \in [0,1]^d$ and nonnegative weights~$w_{j,k} \ge 0$ satisfying~$\sum_{k=1}^{K} w_{j,k} = 1$ for~$j=1,\dots,m$, and define
$g_j(x) = \sum_{k=1}^{K} w_{j,k}\,\mathbf{1}\{x \succeq t_k\}, j=1,\dots,m$,
where~$x \succeq t_k$ is interpreted componentwise.
Because each indicator~$\mathbf{1}\{x \succeq t_k\}$ is monotone nondecreasing in~$x$, every coordinate~$g_j$ is monotone.
The weight normalization ensures~$g(x) \in [0,1]^m$ for every~$x \in [0,1]^d$.

\subsubsection{Lipschitz Problems}
To generate instances that adhere to the Lipschitz structure described in \Cref{sec:lipschitzness}, we draw a random matrix~$A \in \mathbb{R}^{m \times d}$ with spectral norm (i.e., largest singular value)~$\|A\|_2 = L$ and a random phase vector~$b \in \mathbb{R}^m$, and define~$g \colon [0,1]^d \to [0,1]^m$ by
$g_j(x) = \phi\big([Ax+b]_j\big)$, $j \in \set{1,\dots,m}$,
where~$\phi \colon \mathbb{R} \to [0,1]$ is the unit-slope triangle wave
defined by~$\phi(t) = t \bmod 2$ when~$t \bmod 2 \le 1$ and~$\phi(t) = 2 - (t \bmod 2)$ otherwise,
satisfying~$|\phi'(t)| = 1$ almost everywhere. 
At any non-fold point the Jacobian is~$J_g(x) = D(x) A$ with~$D(x) = \operatorname{diag}(\pm 1)$ orthogonal, 
so~$\|J_g(x)\|_2 = \|A\|_2 = L$ almost everywhere.
The fold points form a measure-zero set, and the triangle wave ensures~$g(x) \in [0,1]^m$ for all inputs.
To see that~$g$ is $L$-Lipschitz with respect to the $\ell_2$ norm, note that~$\phi$ is $1$-Lipschitz and acts component-wise, so~$\|g(x) - g(x')\|_2 \leq \|Ax - Ax'\|_2 \leq \|A\|_2\,\|x - x'\|_2 = L\,\|x - x'\|_2$,
and the bound is attained along the top singular vector of~$A$ at almost every point.

\input{Figures/table_random_monotone}
\input{Figures/table_random_lipschitz}

\subsubsection{Results}
In \Cref{table:random_monotone}, we report the cumulative hypervolume difference $\HVD$ on 8 random monotone problems, while in \Cref{table:random_lipschitz}, we report the same metric on 8 random Lipschitz problems with~$L=2$. 
For the cumulative~$\HVD$, a lower value is better, and the worst-case value scales linearly with the number of iterations. 
Each figure represents an average of 100 independent runs. 

On the monotone benchmarks, \Cref{alg:elimination_sampler} achieves the best performance on 7 out of 8 instances, often by a substantial margin.
For example, on instance M3 it improves over the second-best method (\NSGAIII) by roughly a factor of three.
On the Lipschitz benchmarks, \Cref{alg:elimination_sampler} attains the best result on 5 out of 8 instances and the second-best result on 2 of the remaining 3.
Notably, while \MOEAD is competitive on several Lipschitz instances, it exhibits high variance across problems (e.g., strong on L2 and L5 but poor on L8), whereas \Cref{alg:elimination_sampler} delivers consistently low $\HVD$ across all instances.

These results separate two effects.
The improvement over the pure \HALTON baseline shows the value of elimination itself, beyond uniform low-discrepancy coverage of the implementation space.
The additional improvement over the \glspl{acr:ea} shows the benefit of exploiting explicit structural knowledge, such as monotonicity or Lipschitz continuity, that is unavailable to generic population-based heuristics.
Overall, the synthetic experiments support the central claim of the paper: when valid optimistic certificates can be constructed, they translate directly into improved sample efficiency.

\subsection{Realistic Co-Design Problems}

We next consider two realistic problems of intermodal mobility and heterogeneous multi-robot systems co-design.

\subsubsection{Intermodal Mobility System}
\label{sec:intermodal_mobility}
This problem is a realistic co-design problem for intermodal urban mobility, following the model of~\cite{zardini2022co, zardini2023co}.
A central planner jointly selects \gls{acr:av} characteristics in the context of an \gls{acr:amod} service, and public-transit service levels in order to serve a prescribed travel demand while trading off average travel time against total system cost.

The transportation system is modeled as an edge-labeled digraph~$G = \tupIII{V}{A}{c}$ composed of four layers: a road network for \glspl{acr:av}, a road network for micromobility vehicles, a public transit network, and a walking network, connected by mode-switching arcs. 
Customer movements are represented by a set of travel requests~$Q = \set{\tupIII{o_m}{d_m}{ \alpha_m}}_{m=1}^{M}$, where~$o_m, d_m \in V$ are origin--destination pairs and~$\alpha_m > 0$ is the request rate.
A multi-commodity flow model routes customers and rebalancing vehicles across the intermodal network subject to flow conservation and road capacity constraints; we refer the reader to~\cite{zardini2022co} for the full specification.

In the notation of \Cref{def:dpi}, the system-level \gls{acr:dpi} $\dprb_{\mathrm{mob}} \in \dpiOf{\impSetI_{\mathrm{mob}}}{\funPosetF_{\mathrm{mob}}}{\resPosetR_{\mathrm{mob}}}$ is defined as follows.
The functionality \poset is $\funPosetF_{\mathrm{mob}} = \langle \mathcal{Q}, \preceq_\mathcal{Q} \rangle$, the set of serviceable travel demands ordered by inclusion and rate dominance. 
The resource \poset is $\resPosetR_{\mathrm{mob}} = \tupII{\mathbb{R}_{\geq 0}^2}{\preceq}$ with the componentwise order, representing the pair $\tup{t_{\mathrm{avg}}, C_{\mathrm{tot}}}$ of average travel time per trip and total transportation cost (combining fleet acquisition, operational, and infrastructure costs for both \glspl{acr:av} and public transit).
Each implementation~$\I{i} \in \impSetI_{\mathrm{mob}}$ is a tuple $\I{i} = \tup{v_V, n_V, n_S}$, where~$v_V \in [20, 50]$~mph is the \gls{acr:av} achievable speed,~$n_V \in \set{0, \dots,5000}$ is the \gls{acr:av} fleet size, and $n_S \in [0,1]$ is the subway service frequency. Note that we omit the details of tractable implementations.

The \emph{expensive block}~$v_q$ in this co-design diagram is the intermodal mobility system itself: for each candidate~$\tup{v_V, n_V, n_S}$, the resource map~$\reqImp{i} = \tup{t_{\mathrm{avg}}, C_{\mathrm{tot}}}$ is obtained by solving a multi-commodity flow \gls{acr:lp} that minimizes
$t_{\mathrm{avg}} = \frac{1}{\alpha_{\mathrm{tot}}} \sum_{\substack{m \in \{1,\dots,M\}, \langle i,j \rangle \in A}} t_{ij} \cdot f_m(i,j)$,
%
%
over customer flows~$\{f_m\}$ and \gls{acr:av} rebalancing flows~$\{f_V\}$, subject to flow conservation, capacity, and fleet-size constraints. 
Each evaluation of~$v_q$ therefore requires solving a large-scale \gls{acr:lp}, making it the computational bottleneck of the co-design of the intermodal mobility system and the natural target for adaptive sampling via \Cref{alg:elimination_sampler_multigraph}.

\subsubsection{Heterogeneous Multi-Agent System}
\label{sec:heterogeneous_fleet}
This problem is a realistic co-design problem for heterogeneous multi-agent systems performing area coverage and object detection, following~\cite{stralz2026task}. A central designer jointly chooses robot hardware, fleet composition, and a coverage planner to satisfy a prescribed task profile while trading off mission completion time, total cost, and energy consumption.

The system is represented by a compositional co-design diagram with four interconnected design blocks: a \emph{fleet composer}, a \emph{planner}, an \emph{executor}, and an \emph{evaluator}~\cite{stralz2026task}. The fleet may contain three robot types: a medium-sized aerial robot, a large aerial robot, and a ground robot. Each robot type is defined by an actuation module, a battery module chosen from eight chemistries, and a fixed sensing module. The actuation module determines speed, acceleration, and turning radius; the battery chemistry determines specific energy density and energy-cost density; and the sensing module determines sensing radius and target detection accuracy. The planner is selected from a catalog of area-coverage algorithms.

In the notation of \Cref{def:dpi}, the system-level \gls{acr:dpi} $\dprb_{\mathrm{het}} \in \dpiOf{\impSetI_{\mathrm{het}}}{\funPosetF_{\mathrm{het}}}{\resPosetR_{\mathrm{het}}}$ is defined as follows.
The functionality \poset is $\funPosetF_{\mathrm{het}} = \langle \mathcal{T}, \preceq_{\mathcal{T}} \rangle$, the set of task profiles ordered by dominance, where each task profile specifies minimum coverage-percentage thresholds over one or more search spaces.
The resource \poset is $\resPosetR_{\mathrm{het}} = \tupII{\mathbb{R}_{\geq 0}^2}{\preceq}$ with the componentwise order, representing the tuple $\tup{t_{\mathrm{mission}}, C_{\mathrm{tot}}}$ of mission completion time and total fleet acquisition cost.
Each~$\I{i} \in \impSetI_{\mathrm{het}} \subseteq \mathbb{N}^3$ is a tuple $\I{i} = \tupIII{n_{1}}{n_{2}}{n_{3}}$ specifying the number of deployed robots per type.

The \emph{expensive block}~$v_q$ in this co-design diagram is the planner--executor simulation pipeline: for each candidate~$\I{i}$, the resource map~$\reqImp{i} = \tup{t_{\mathrm{mission}}, C_{\mathrm{tot}}}$ is obtained by executing the selected planner to generate waypoint assignments, propagating these through a physics-based trajectory simulator that computes dynamically feasible paths and evaluating area-coverage metrics over the resulting trajectory collection.
Each evaluation of~$v_q$ therefore requires running a complete multi-robot simulation, making it the computational bottleneck of the system co-design.

\subsubsection{Case Study Details}
For both problems in \Cref{sec:intermodal_mobility,sec:heterogeneous_fleet}, the optimal nondominated resource sets are intractable to compute exactly, and robust empirical comparisons would require repeated evaluations of expensive simulation models. \
We therefore discretize the implementation spaces to $\card{\impSetI_{\mathrm{mob}}}=2197$ and $\card{\impSetI_{\mathrm{het}}}=1000$, enabling cached simulation outputs.

For the \glspl{acr:ea}, we tune over population sizes $\set{50,100,200}$, distribution indices $\set{5,10,15}$, and crossover probabilities $\set{0.1,0.5,0.9}$.
For \qLogEHVI, the Gaussian process fitting and acquisition function optimization are aware of the noiseless setup.
Our method uses the relaxed monotone optimistic evaluation of \Cref{sec:monotonicity}, with uniform forced acceptance probability $\delta=0.05$. 
For $\dprb_{\mathrm{mob}}$, the intuition behind monotonicity is that $t_{\mathrm{avg}}$ likely decreases with a larger \gls{acr:av} fleet while $C_{\mathrm{tot}}$ increases, and analogous reasoning applies to the remaining design variables in $\dprb_{\mathrm{mob}}$ and for $\dprb_{\mathrm{het}}$.
Because monotonicity may not hold exactly—for instance, larger fleets can increase mission time $t_{\mathrm{mission}}$ through more complex trajectories—we use a more conservative elimination rule, replacing that maximum of~\eqref{eq:monotonicity_opt_set} by the average of its $K$ largest elements for $K \in \set{1,2,5,10,15}$. 
We also cap the rejection-sampling loop in \Cref{alg:elimination_sampler_multigraph} at~50 iterations.

\subsubsection{Results}

\input{Figures/intermodal_mobility}
\input{Figures/heterogeneous_coverage}

\Cref{fig:intermodal_mobility,fig:heterogeneous_coverage} report the $\HVD$ as a function of the number of iterations for the intermodal mobility and the heterogeneous multi-agent co-design problems, respectively.
Each curve is an average of 500 independent runs (50 for \qLogEHVI), and the shaded regions represent half the standard deviations.
In both settings, \Cref{alg:elimination_sampler_multigraph} recovers the exact antichain substantially faster than all baselines, reaching low $\HVD$ values with far fewer evaluations of the expensive block~$v_q$.

The two problems exhibit notably different baseline behaviors.
In \Cref{fig:heterogeneous_coverage}, the low-discrepancy \HALTON sampler performs competitively, suggesting that the true Pareto antichain is well-dispersed across the implementation space, so quasi-random space-filling already yields adequate coverage.
Nevertheless, our method still improves upon \HALTON, reaching an $\HVD$ of $10^{-5}$ in roughly half the iterations.
In \Cref{fig:intermodal_mobility}, by contrast, \HALTON performs poorly and is outpaced by several \glspl{acr:ea}, indicating that the antichain is concentrated in a small region of the implementation space that space-filling heuristics are slow to locate.
Our approach exploits the (approximate) monotone structure of the co-design diagram to focus evaluations precisely on that region: it requires approximately $33\%$ fewer samples than \RVEA---the best-performing baseline---to reach an $\HVD$ of $10^{-3}$, and approximately $60\%$ fewer samples to reach $10^{-5}$.
Further, although \qLogEHVI converges faster than the other methods early on, it eventually plateaus and struggles to recover the exact antichain.
Finally, in \Cref{fig:heterogeneous_coverage}, we report performance when the baseline methods optimize over the full implementation space of~$\cdprb$.
We observe that including tractable implementations in the decision space of the \glspl{acr:ea} drastically degrades their performance.

\input{Figures/table_ablation}

To isolate the effect of elimination under approximate monotonicity, we report performance metrics for our approach with and without relaxed rejections on the intermodal mobility benchmark in~\Cref{table:ablation}, where exact recovery means the true antichain of the discretized benchmark is identified up to machine precision. 
The results show that our approach degrades gracefully under mild model misspecification, and that relaxed rejection substantially improves recovery probability at the cost of a slight reduction in convergence speed.

\input{Figures/table_time}

Finally, \Cref{table:time} reports the per-step CPU time of \Cref{alg:elimination_sampler_multigraph}, \HALTON, \NSGAIII, \qLogEHVI and the mobility \gls{acr:lp}~$v_q$. 
Our method is slower than \HALTON and \NSGAIII due to the rejection-sampling loop, with roughly a twofold increase in average runtime and up to an order-of-magnitude increase in the worst case.
Nevertheless, even its worst-case CPU time is still two orders of magnitude smaller than the minimum time needed to solve a single mobility~\gls{acr:lp} and lower than the \gls{acr:bo} baseline, making the additional overhead worthwhile given the improvement in convergence.

%% file: Figures/table_random_monotone.tex
\begin{table*}
\setlength{\abovecaptionskip}{0pt}
\centering
\footnotesize
\renewcommand{\arraystretch}{\TableVerStretch}
\caption{
The cumulative hypervolume difference $\HVD$ ($\downarrow$) on randomly generated monotone problems with $\impSetI=\mathbb{R}^3$, $\funPosetF=\mathbb{R}^2$, and $\resPosetR=\mathbb{R}^2$ after 4000 iterations.
The best-performing method is shown in \textbf{bold}, and the second-best is \underline{underlined}.
}
\begin{adjustbox}{max width=\textwidth, center}
\begin{NiceTabular}{c|cccccccc}
    \toprule
    \textbf{Method} & M1 & M2 & M3 & M4 & M5 & M6 & M7 &  M8 \\
    \midrule
    \OPTELIM & \bm{$8.81\times10^{1}$} & \bm{$1.94\times10^{2}$} & \bm{$1.14\times10^{2}$} & \bm{$1.78\times10^{2}$} & \bm{$4.97\times10^{1}$} & \bm{$4.87\times10^{1}$} & $3.26\times10^{2}$      & \bm{$3.46\times10^{1}$} \\
    \HALTON  & $3.48\times10^{2}$      & $5.97\times10^{2}$      & $6.88\times10^{2}$      & $4.41\times10^{2}$      & $3.35\times10^{2}$      & $1.89\times10^{2}$      & $7.10\times10^{2}$      & $1.82\times10^{2}$ \\
    \NSGAIII & \ul{$1.45\times10^{2}$} & \ul{$3.35\times10^{2}$} & \ul{$3.78\times10^{2}$} & \ul{$3.14\times10^{2}$} & \ul{$1.91\times10^{2}$} & $1.20\times10^{2}$      & \ul{$2.28\times10^{2}$} & $7.48\times10^{1}$ \\
    \MOEAD   & $1.57\times10^{2}$      & $3.41\times10^{2}$      & $6.05\times10^{2}$      & $3.82\times10^{2}$      & $2.34\times10^{2}$      & $1.59\times10^{2}$      & \bm{$1.64\times10^{2}$} & \ul{$5.93\times10^{1}$} \\
    \RVEA    & $4.77\times10^{2}$      & $7.04\times10^{2}$      & $1.29\times10^{3}$      & $4.04\times10^{2}$      & $5.65\times10^{2}$      & \ul{$1.06\times10^{2}$} & $2.93\times10^{2}$      & $1.21\times10^{2}$ \\
    \KGB     & $1.69\times10^{2}$      & $3.78\times10^{2}$      & $4.83\times10^{2}$      & $3.47\times10^{2}$      & $2.25\times10^{2}$      & $1.37\times10^{2}$      & $3.06\times10^{2}$      & $8.56\times10^{1}$ \\
    \bottomrule
\end{NiceTabular}
\end{adjustbox}
\label{table:random_monotone}
\end{table*}

%% file: Figures/table_random_lipschitz.tex
\begin{table*}
\setlength{\abovecaptionskip}{0pt}
\centering
\footnotesize
\renewcommand{\arraystretch}{\TableVerStretch}
\caption{
The cumulative hypervolume difference $\HVD$ ($\downarrow$) on randomly generated Lipschitz problems with $\impSetI=\mathbb{R}^4$, $\funPosetF=\mathbb{R}^2$, and $\resPosetR=\mathbb{R}^2$ after 2000 iterations.
The best-performing method is shown in \textbf{bold}, and the second-best is \underline{underlined}.
}
\begin{adjustbox}{max width=\textwidth, center}
\begin{NiceTabular}{c|cccccccc}
    \toprule
    \textbf{Method} & L1 & L2 & L3 & L4 & L5 & L6 & L7 & L8 \\
    \midrule
    \OPTELIM & \bm{$5.89\times10^{1}$} & $1.85\times10^{1}$      & \bm{$4.92\times10^{1}$} & \bm{$2.26\times10^{1}$} & \ul{$3.80\times10^{1}$} & \ul{$1.79\times10^{1}$} & \bm{$4.63\times10^{1}$} & \bm{$4.20\times10^{1}$} \\
    \HALTON  & $2.84\times10^{2}$      & $2.77\times10^{1}$      & $2.55\times10^{2}$      & \ul{$2.52\times10^{2}$} & $7.04\times10^{1}$      & $2.94\times10^{1}$      & $9.85\times10^{1}$      & \ul{$2.96\times10^{2}$} \\
    \NSGAIII & $1.82\times10^{2}$      & $2.00\times10^{1}$      & $1.29\times10^{2}$      & $3.54\times10^{2}$      & $4.93\times10^{1}$      & $2.12\times10^{1}$      & $6.37\times10^{1}$      & $8.21\times10^{2}$      \\
    \MOEAD   & $3.13\times10^{2}$      & \bm{$1.39\times10^{1}$} & $1.25\times10^{2}$      & $3.87\times10^{2}$      & \bm{$3.35\times10^{1}$} & \bm{$1.59\times10^{1}$} & \ul{$4.71\times10^{1}$} & $1.04\times10^{3}$      \\
    \RVEA    & \ul{$1.79\times10^{2}$} & \ul{$1.76\times10^{1}$} & \ul{$1.18\times10^{2}$} & $3.34\times10^{2}$      & $4.46\times10^{1}$      & $2.03\times10^{1}$      & $5.66\times10^{1}$      & $1.01\times10^{3}$      \\
    \KGB     & $2.37\times10^{2}$      & $2.02\times10^{1}$      & $1.55\times10^{2}$      & $4.39\times10^{2}$      & $4.95\times10^{1}$      & $2.17\times10^{1}$      & $7.04\times10^{1}$      & $9.57\times10^{2}$      \\
    \bottomrule
\end{NiceTabular}
\end{adjustbox}
\label{table:random_lipschitz}
\end{table*}

%% file: Figures/intermodal_mobility.tex
\begin{figure}[tb]
\centering
\setlength{\abovecaptionskip}{0pt}
\includegraphics[width=0.85\columnwidth]{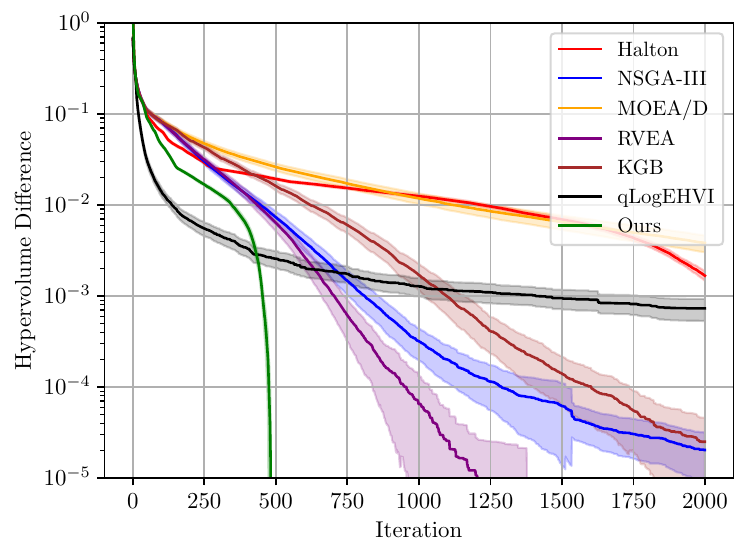}
\caption{
Hypervolume difference $\HVD$ as a function of the number of iterations for the intermodal mobility co-design problem. 
}
\label{fig:intermodal_mobility}
\end{figure}

%% file: Figures/heterogeneous_coverage.tex
\begin{figure}[tb]
\centering
\setlength{\abovecaptionskip}{0pt}
\includegraphics[width=0.85\columnwidth]{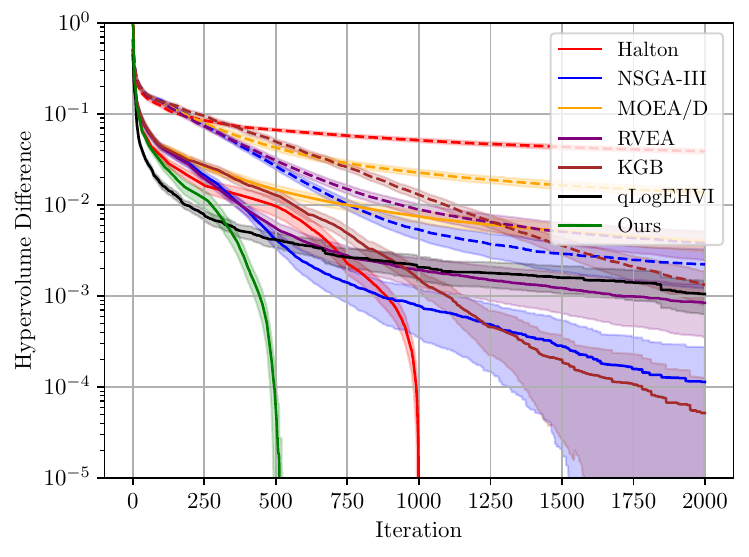}
\caption{
Hypervolume difference $\HVD$ as a function of the number of iterations for the heterogeneous multi-agent system co-design problem.
Dashed lines show performance when methods optimize over all implementations in~$\cdprb$.
}
\label{fig:heterogeneous_coverage}
\end{figure}

%% file: Figures/table_ablation.tex
\begin{table}[t]
\setlength{\abovecaptionskip}{0pt}
\centering
\footnotesize
\renewcommand{\arraystretch}{\TableVerStretch}
\caption{Ablation on relaxed rejections in the intermodal mobility case study. All metrics are evaluated at iteration $N=2000$.}
\begin{adjustbox}{max width=\columnwidth, center}
\begin{NiceTabular}{c|cccc}
    \toprule
    \textbf{Metric} & \Block{1-1}{\OPTELIM \\ ($\delta=0, K=1$)} & \Block{1-1}{\OPTELIM \\ ($\delta=0.05, K=2$)} & \RVEA & \HALTON \\
    \midrule
    \textbf{Cumulative \bm{$\HVD$}} ($\downarrow$) & $18.53$ & $19.95$ & $27.27$ & $41.61$ \\
    \textbf{Exact Recovery} ($\uparrow$) & $89.8\%$ & $100\%$ & $99.8\%$ & $0\%$ \\
    \bottomrule
\end{NiceTabular}
\end{adjustbox}
\label{table:ablation}
\end{table}

%% file: Figures/table_time.tex
\begin{table}[tb]
\setlength{\abovecaptionskip}{0pt}
\centering
\footnotesize
\renewcommand{\arraystretch}{\TableVerStretch}
\caption{CPU time per step for \Cref{alg:elimination_sampler_multigraph}, \HALTON, \NSGAIII, \qLogEHVI, and the mobility \gls{acr:lp}~$ v_q$.}
\begin{adjustbox}{max width=\columnwidth, center}
\begin{NiceTabular}{c|cccc|c}
    \toprule
    \textbf{Time} & \OPTELIM & \HALTON & \NSGAIII & \qLogEHVI & Sim. $v_q$ \\
    \midrule
    \textbf{Avg.} & $2.61\times10^{-3}$ & $0.92\times10^{-3}$ & $1.07\times10^{-3}$ & $3.88\times10^{-1}$ & $5.84\times10^{1}$ \\
    \textbf{Min.} & $0.16\times10^{-3}$ & $0.15\times10^{-3}$ & $0.17\times10^{-3}$ & $1.34\times10^{-1}$ & $5.22\times10^{0}$ \\
    \textbf{Max.} & $5.23\times10^{-2}$ & $2.01\times10^{-3}$ & $2.19\times10^{-3}$ & $9.57\times10^{-1}$ & $1.70\times10^{2}$ \\
    \bottomrule
\end{NiceTabular}
\end{adjustbox}
\label{table:time}
\end{table}

%% file: Sections/conclusion.tex
\section{Conclusions}
\label{sec:conclusion}

We developed an online learning approach for multi-objective decision-making in monotone co-design that uses history-dependent optimistic bounds on unknown functionality and resource mappings to rule out unpromising implementations before incurring costly evaluations. 
The method is guaranteed not to eliminate any design that could still improve the current target-feasible Pareto set, and its feasible search region shrinks monotonically as more data are collected. 
We showed how this principle can be instantiated under several useful structural assumptions, and extended it to compositional co-design problems by propagating subsystem-level optimistic bounds through multigraph-structured design models to obtain system-level certificates with the same correctness guarantees. 
Across synthetic and application-driven benchmarks, including mobility and multi-robot co-design, the resulting algorithms achieved comparable solution quality with far fewer expensive evaluations than quasi-random search and leading multi-objective baselines.

Several directions remain open.
First, deriving quantitative, problem-dependent sample complexity bounds that characterize how fast the admissible set shrinks would deepen the theoretical understanding of the statistical complexity of system design. 
This would entail not only understanding the complexity of learning a single block, but also characterizing how the compositional completion of the co-design diagram amplifies or attenuates that complexity.
Second, extending the compositional framework to settings with multiple expensive blocks---requiring coordinated exploration across several unknown subsystems---is an important step toward broader applicability.
Third, incorporating stochastic or noisy evaluations through high-probability confidence bounds would expand the framework to experimental settings where exact observations are unavailable.
Fourth, more sophisticated acquisition strategies---such as information-directed sampling or priority rules derived from pessimistic evaluators---may yield noticeable sample-efficiency gains, particularly in high-dimensional implementation spaces.